\documentclass{article}
\def\Bbb{\mathbb}
\def\cal{\mathcal}
\usepackage{color}
\usepackage{amsthm,amssymb,amsfonts,mathrsfs}
\usepackage{amsmath}
\catcode`@=11 \@addtoreset{equation}{section}

\catcode`@=12
\newtheorem{Theorem}{Theorem}[section]
\newtheorem{Proposition}{Proposition}[section]
\newtheorem{Lemma}{Lemma}[section]
\newtheorem{Corollary}{Corollary}[section]
\theoremstyle{definition}

\newtheorem{Definition}{Definition}[section]
\newtheorem{Remark}{Remark}

\newtheorem{Assumptions}{Hypothesis}[section]
\def\R{{\mathbb{R}}}

\def\vp{\varphi}
\def\ve{\varepsilon}

\title{Carleman estimates and observability inequalities for parabolic equations with interior degeneracy}
\author{{\sc Genni Fragnelli\thanks{Research supported by the GNAMPA project {\sl Equazioni di evoluzione degeneri e singolari: controllo e applicazioni}}}\\
Dipartimento di Matematica\\ Universit\`{a} di Bari "Aldo Moro"\\
Via
E. Orabona 4\\ 70125 Bari - Italy\\ email: genni.fragnelli@uniba.it\\
{\sc Dimitri Mugnai}\\
Dipartimento di Matematica e Informatica\\Universit\`a di
Perugia\\Via Vanvitelli 1, 06123 Perugia - Italy\\ email:
mugnai@dmi.unipg.it}

\date{}
\begin{document}

\maketitle

\begin{abstract}
We consider a parabolic problem with degeneracy in the interior of
the spatial domain, and we focus on Carleman estimates for the
associated adjoint problem. The novelty of interior degeneracy does
not let us adapt previous Carleman estimate to our situation. As an
application, observability inequalities are established.
\end{abstract}
%%%%%%%%%%%%%%%%%%%%%%%%%%%%%%%%%%%%
%%%%%%%%%%%%%%%%%%%%%%%%%%%%%%%%%%%
%%%%%%%%%%%%%%%%%%%%%%%%%%%%%%%%%%%%
%%%%%%%%%%%%%%%%%%%%%%%%%%%%%%%%%%%%
%%%%%%%%%%%%%%%%%%%%%%%%%%%%%%%%%%%%
%%%%%%%%%%%%%%%%%%%%%%%%%%%%%%%%%%%%%%%%
%%%%%%%%%%%%%%%%%%%%%%%%%%%%%%%%%%%%%%%



\section{Introduction}

In this paper we focus on two subjects that in the last years have
been object of a large number of papers, i.e. {\em degenerate}
problems and {\em Carleman estimates}. Indeed, as pointed out by
several authors, many problems coming from physics (boundary layer models in \cite{br}, models of Kolmogorov type in \cite{bz},
 models of Grushin type in \cite{bcg}, \ldots), biology (Wright-Fisher models in \cite{s} and Fleming-Viot models in \cite{fv}), and economics (Black-Merton-Scholes equations
in \cite{EG}) are described by degenerate
parabolic equations.

On the other hand, the fields of applications of Carleman estimates
are so wide that it is not surprising that also several papers are
concerned with such a topic. A first example is their application to
{\em global null controllability} (see \cite{m0}, \cite{m},
\cite{acf}, \cite{beau}, \cite{cfr}, \cite{cfr1}, \cite{cfv},
\cite{fcgb}, \cite{fdt}, \cite{f}, \cite{lrr}, \cite{lr},
\cite{LRL}, \cite{Ra}, \cite{vz} and the references therein): for
all $T>0$ and for all initial data $u_0\in L^2((0,T) \times(0,1))$
there is a suitable control $h\in L^2((0,T) \times(0,1))$ such that
the solution of
\begin{equation}\label{0}
\begin{cases}u_t - \left(au_x \right) _x  =h(t,x) \chi_{\omega}(x), \quad
(t,x) \in (0,T) \times(0,1),\\
u(0,x)=u_0(x)
\end{cases}
\end{equation}
with some boundary conditions, satisfies $ u(T, x)=0 $ for all $x
\in [0,1]$. Here $\chi_\omega$ denotes, as usual, the characteristic
function of the set $\omega$, i.e. $\chi(x)=1$ if $x\in \omega$ and
$\chi(x)=0$ if $x\not\in\omega$.

Moreover, Carleman estimates may be a fundamental tool in inverse
problems, in parabolic, hyperbolic and fractional settings, e.g. see
\cite{chue}, \cite{ima}, \cite{salo}, \cite{salo1}, \cite{Wu},
\cite{Ya} and their references.

However, all the previous papers deal with problems like \eqref{0}
which are non degenerate or admit that the function $a$
degenerates at the boundary of the domain, for example
\[
a(x)= x^k(1-x)^\alpha, \quad x \,\in \,[0,1],
\]
where $k$ and $\alpha$ are positive constants.

To our best knowledge, \cite{s1} is the first paper treating a
problem with a degeneracy which may occur in the interior of the
spatial domain. In particular, Stahel considers a parabolic problem
in a bounded smooth domain $\Omega$ of $\R^N$ with Dirichlet,
Neumann or mixed boundary conditions, associated to a $N\times N$
matrix $a$, which is positive definite and symmetric, but whose
smallest eigenvalue might converge to $0$ as the space variable
approaches a singular set contained in the closure of the spatial
domain. In this case, he proves that the corresponding abstract
Cauchy problem has a solution, provided that $\underline{a}^{-1} \in
L^q(\Omega, \R)$ for some $q >1$, where
\[ \underline{a}(x):= \min \{a(x)\xi\cdot \xi: \|\xi \|=1\}.
\]

Moreover, while  in \cite{s1} only the existence of a solution for
the parabolic problem is considered, in \cite{fggr} the authors
analyze in detail the degenerate operator
\[
{\cal A}u:= (au_x)_x
\]
in the space $L^2 (0,1)$, with or without weight, proving that in
some cases it is nonpositive and selfadjoint, hence it generates a
cosine family and, as a consequence, an analytic semigroup. In
\cite{fggr} the well-posedness of \eqref{0} with Dirichlet boundary
conditions is also treated, but nothing is said about
controllability properties.

This paper is then concerned with several inequalities (Carleman
estimates, observability inequalities, Hardy--Poincar\'e
inequalities) related to the parabolic equation with interior
degeneracy
\begin{equation}\label{linear}
\begin{cases}
u_t - \left(au_x \right) _x  =h(t,x) \chi_{\omega}(x),\quad
(t,x) \in (0,T) \times(0,1),\\
u(t,0)=u(t,1)=0, \\
u(0,x)=u_0(x),
\end{cases}
\end{equation}
where $(t,x) \in Q_T:=(0,T) \times (0,1)$, $u_0 \in L^2(0,1)$, $a$
degenerates at $x_0\in (0,1)$ and the control $h \in L^2(Q_T)$ acts
on a nonempty subdomain $\omega$ of $(0,1)$ such that
\[
x_0\in \omega.
\]

We shall admit two types of degeneracy for $a$, namely weak and
strong degeneracy. In particular, we make the following assumptions:
\begin{Assumptions}\label{Ass0}
{\bf Weakly degenerate case (WD):} there exists $x_0 \in (0,1)$ such
that $a(x_0)=0$, $a>0$ on $[0, 1]\setminus \{x_0\}$, $a\in
C^1\big([0,1]\setminus \{x_0\}\big)$ and there exists $K \in (0,1)$
such that $(x-x_0)a' \le K a$ in $[0,1]\setminus \{x_0\}$.
\end{Assumptions}

\begin{Assumptions}\label{Ass01}
{\bf Strongly degenerate case (SD):} there exists $x_0 \in (0,1)$
such that $a(x_0)=0$, $a>0$ on $[0, 1]\setminus \{x_0\}$, $a\in
C^1\big([0,1]\setminus \{x_0\}\big)\cap W^{1,\infty}(0,1)$ and there
exists $K \in [1,2)$ such that $(x-x_0)a' \le K a$ in
$[0,1]\setminus \{x_0\}$.
\end{Assumptions}
Typical examples for weak and strong degeneracies are $a(x)=|x-
x_0|^{\alpha}, \; 0<\alpha<1$ and $a(x)= |x- x_0|^{\alpha}, \;
1\le\alpha<2$), respectively.

The starting point of the paper is actually the analysis of the
adjoint problem to \eqref{linear}
\begin{equation}\label{01}
\begin{cases}
v_t + \left(a(x)v_x \right) _x =h, & (t,x) \in (0,T) \times (0,1),\\
v(t,1)=v(t,0)=0, &  t \in (0,T).\\
\end{cases}
\end{equation}
In particular, for any (sufficiently regular) solution $v$ of such a
system we derive the new fundamental Carleman estimate having the
form
\begin{equation}\label{car}
\begin{aligned}
&\int_0^T\int_0^1 \left(s\Theta a(v_x)^2 + s^3 \Theta^3
\frac{(x-x_0)^2}{a} v^2\right)e^{2s\varphi}dxdt\\
&\le C\left(\int_0^T\int_0^1 h^{2}e^{2s\varphi}dxdt +
s\int_0^T\left[a\Theta e^{2s \varphi}(x-x_0)(v_x)^2
dt\right]_{x=0}^{x=1}\right),
\end{aligned}
\end{equation}
for all $s \ge s_0$, where $s_0$ is a suitable constant. Here
$\Theta(t) :=[t(T-t)]^{-4}$, and $\varphi(t,x):= \Theta(t)\psi(x),$
with $\psi(x) <0$ given explicitly in terms of $a$, see \eqref{c_1}.
Of course, for the Carleman inequality, the location of $x_0$ with
respect to the control set $\omega$ is irrelevant, since $\omega$
plays no role at all.

For the proof of the previous Carleman estimate a fundamental
role is played by the second basic result of this paper, that is
a general Hardy-Poincar\'e type inequality proved in Proposition
\ref{HP}, of independent interest, that we establish for all
functions $w$ which are only {\em locally absolutely continuous} in
$[0,1]\setminus \{x_0\}$ and such that
\[
w(0)=w(1)=0, \text{ and }\int_0^1 p(x)|w'(x)|^2dx < +\infty,
\]
and which reads
\[
\int_0^1\frac{p(x)}{(x-x_0)^2}w^2(x)dx \le C\int_0^1
p(x)|w'(x)|^2dx.
\]

Here $p$ is {\em any} continuous function in $[0,1]$, with $p>0$ on
$[0,1]\setminus \{x_0\}$, $p(x_0)=0$ and there exists $q \in (1,2)$
such that the function
$$
\begin{aligned}
& x \longrightarrow \dfrac{p(x)}{|x-x_0|^{q}} \mbox { is
nonincreasing on the left of } x=x_0 \\& \mbox{and nondecreasing on
the right of } x=x_0.
\end{aligned}
$$

Applying estimate \eqref{car} to any solution $v$ of the adjoint
problem \eqref{01} with $h=0$, we shall obtain the observability
inequality
\[
\int_0^1v^2(0,x) dx \le C\int_0^T\int_{\omega} v^2(t,x)dxdt,
\]
where now we consider the fact that $x_0\in \omega$.

Such a result is then extended to the complete linear problem
\begin{equation}\label{nlo}
\begin{cases}
    u_t - \left(a(x)u_x \right) _x  + c(t,x)u  =h(t,x) \chi_{\omega}(x), & \ (t,x) \in (0,T) \times (0,1), \\
    u(t,1)=u(t,0)=0, & \  t \in (0,T),\\
     u(0,x)=u_0(x) , & \  x \in (0,1),
\end{cases}
\end{equation}
where $c$ is a bounded function, previously proving a Carleman
estimate associated to this problem, see Corollary \ref{cor_c} and
Proposition \ref{obser_c}.

Finally, observe that on $a$ we require that there exists $K \in
(0,2)$ such that $(x-x_0)a' \le K a$ in $[0,1]$, and $K\geq 2$ is
excluded. This technical assumption, which is essential in all our
results, is the same made, for example, in \cite{acf}, where the
degeneracy occurs at the boundary of the domain and the problem
fails to be null controllable on the whole interval $[0,1]$. But
since the null controllability for the parabolic problem and the
observability inequality for the adjoint problem are equivalent
(\cite{LRL}), it is not surprising that we require $K<2$.

\medskip

The paper is organized as follows. In Section \ref{sec2} we give the
precise setting for the weak and the strong degenerate cases and
some general tools we shall use several times; in particular a
general weighted Hardy--Poincar\'e inequality is established. In
Section \ref{Carleman estimate} we provided the main result of this
paper, i.e. a new Carleman estimate for degenerate operators with
interior degeneracy. In Section \ref{sec4} we apply the previous
Carleman estimates together with a Caccioppoli type inequality to
prove an observability inequality. Finally, in Section \ref{sec5} we
extend the previous results to complete linear problems.

We conclude this introduction with the following
\begin{Remark}
At a first glance, one may think that our results can be obtained
just by a ``translation'' of the ones obtained in \cite{acf}, but
this is not the case. Indeed, in \cite{acf} the degeneracy point was
the origin, where the authors put suitable homogeneous boundary
conditions (Dirichlet for the (WD) case and weighted Neumann in the
(SD) case) which coincide exactly with the ones obtained by the
characterizations of the domains of the operators given in
Propositions \ref{characterization} and \ref{domain} below. In this
way they can control {\em a priori} the possible uncontrolled
behaviour of the solution at the degeneracy point, while here we
don't impose any condition on the solution in the interior point
$x_0$.
\end{Remark}

\section{Preliminary Results}\label{sec2}
In order to study the well-posedness of problem \eqref{linear}, we
introduce the operator
\[
{\cal A}u:=(au_x)_x
\]
and we consider two different classes of weighted Hilbert spaces,
which are suitable to study two different situations, namely the
{\em weakly degenerate} (WD) and the {\em strongly degenerate} (SD)
cases:

{\bf CASE (WD):}
\[
\begin{aligned}
 H^1_a (0,1):=\big\{&u \text{ is absolutely
continuous in } [0,1],
\\ & \sqrt{a} u' \in  L^2(0,1) \text{ and } u(0)=u(1)=0
\big\},
\end{aligned}
\]
and
\[
\label{Ha2} \qquad H^2_a(0,1) :=  \big\{ u \in H^1_a(0,1) |\,au' \in
H^1(0,1)\big\};
\]

{\bf CASE (SD):}

\[
\begin{aligned}
H^1_a(0,1):=\big\{ u \in L^2(0,1) \ \mid  \,&u \text{ locally
absolutely continuous in } [0,x_0) \cup (x_0,1], \\ & \sqrt{a} u'
\in L^2(0,1) \text{ and } u(0)= u(1)=0 \big\}
\end{aligned}
\]
and
\[
\label{Ha2S} \qquad H^2_a(0,1) :=  \big\{ u \in H^1_a(0,1) |\,au'
\in H^1(0,1)\big\}.
\]
In both cases we consider the norms
\[
\|u\|^2_{H^1_a(0,1)}:= \|u\|^2_{L^2(0,1)} +
\|\sqrt{a}u'\|^2_{L^2(0,1)},\] and
\[\|u\|^2_{H^2_a(0,1)} :=
\|u\|^2_{H^1_a(0,1)} + \|(au')'\|^2_{L^2(0,1)}
\]
and we set
\[
D({\cal A})=H^2_a(0,1).
\]

The function $a$ playing a crucial role, it is non surprising
that the following lemma is crucial as well:
\begin{Lemma} \label{rem}
Assume that Hypothesis $\ref{Ass0}$ or $\ref{Ass01}$ is satisfied.
\begin{enumerate}
\item Then for all $\gamma \ge K$ the map
$$
\begin{aligned}
& x \mapsto \dfrac{|x-x_0|^\gamma}{a} \mbox { is nonincreasing on
the left of } x=x_0 \\
& \mbox{and nondecreasing on the right of }
x=x_0,\\
&\mbox{ so that }\lim_{x\to x_0}\dfrac{|x-x_0|^\gamma}{a}=0 \mbox{
for all }\gamma>K.
\end{aligned}
$$
\item If $K<1$, then
    $\displaystyle\frac{1}{a} \in L^{1}(0,1)$.\\
\item If $K \in[1,2)$, then $\displaystyle \frac{1}{\sqrt{a}} \in
    L^{1}(0,1)$ and $\displaystyle \frac{1}{a}\not \in L^1(0,1)$.
\end{enumerate}
\end{Lemma}
\begin{proof} The first point is an easy consequence of the
assumption. Now, we prove  the second point: by the first part, it
follows that
 $$
\displaystyle\frac{|x-x_0|^{K}}{a(x)} \le
\max\left\{\displaystyle\frac{x_0^K}{a(0)},\displaystyle\frac{(1-x_0)^K}{a(1)}\right\}.
$$ Thus
\[
\frac{1}{a(x)} \le
\max\left\{\frac{x_0^K}{a(0)},\frac{(1-x_0)^K}{a(1)}\right\}\frac{1}{|x-x_0|^K}.
\]
Since $K<1$, the right-hand side of the last inequality is
integrable, and then $\displaystyle \frac{1}{a} \in L^{1}(0,1)$.
Analogously, one obtains the third point.

On the contrary, the fact that $a\in C^1([0,1])$, and $\displaystyle
\frac{1}{\sqrt{a}} \in L^{1}(0,1)$ implies that $\displaystyle
\frac{1}{a}\not \in L^1(0,1)$. Indeed, the assumptions on $a$ imply
that $a(x) = \displaystyle \int_{x_0}^xa'(s)ds\le C |x-x_0|$ for a
positive constant $C$. Thus for all $x\neq x_0$, $\displaystyle
\frac{1}{a(x)} \ge C \frac{1}{|x-x_0|} \not \in L^1(0,1)$.
\end{proof}

We immediately start using the lemma above, giving the following
characterizations for the (SD) case which are already given in
\cite[Propositions 2.3 and 2.4]{fggr}, but whose proofs we repeat
here to make precise some calculations.
\begin{Proposition}\label{characterization}
Let
\[
\begin{aligned} X:=\{ u \in L^2(0,1)\,| \ &u \text{ locally
absolutely continuous in } [0,1]\setminus \{x_0\},  \\ &\sqrt{a}u'
\in L^2(0,1), au \in H^1_0(0,1) \;\text{and} \\&
(au)(x_0)=u(0)=u(1)=0\}.
\end{aligned}
\]
Then, under Hypothesis $\ref{Ass01}$ we have
\[
H^1_a(0,1)=X.
\]
\end{Proposition}
\begin{proof}
Obviously, $X \subseteq H^1_a$. Now we take $u \in H^1_a$, and we
prove that $u \in X$. First, observe that $(au)(0)= (au)(1)=0$.
Moreover, since $a \in W^{1,\infty}(0,1)$, then $(au)' = a' u+ au' \in
L^2(0,1)$. Thus, for $x <x_0$, one has
\[
(au)(x)= \int_0^x(au)'(t)dt
\]
This implies that there exists $\lim_{x \rightarrow
x_0^-}(au)(x)=(au)(x_0) =\int_0^{x_0}(au)'(t)dt= L \in \R$. If $L
\neq 0$, then there exists $C>0$ such that
\[
|(au)(x)| \ge C
\]
for all $x$ in a neighborhood of $x_0$, $x\neq x_0$. Thus, setting
$C_1:=\displaystyle \frac{C^2}{\max_{[0,1]}a(x)} > 0$, it follows
that
\[
|u^2(x)| \ge \frac{C^2}{a^2(x)}\ge \frac{C_1}{a(x)},
\]
for all $x$ in a left neighborhood of $x_0$, $x\neq x_0$. But, since
the operator is strongly degenerate, $\displaystyle \frac{1}{a} \not
\in L^1(0,1)$ thus $u \not \in L^2(0,1)$. Hence $L=0$. Analogously,
starting from
\[
(au)(x)=-\int_x^1(au)'(t)dt \quad \mbox{ for }x>x_0,
\]
one can prove that $\lim_{x \rightarrow x_0^+}(au)(x)=(au)(x_0) =0$
and thus $(au)(x_0)=0$. From this it also easily follows that
$(au)'$ is the distributional derivative of $au$, and so $au\in
H^1_0(0,1)$, i.e. $u\in X$.
\end{proof}

Using the previous result, one can prove the following additional
characterization.
\begin{Proposition}\label{domain}
Let
\[
\begin{aligned}
D:=\{ u \in L^2(0,1) \ & \ u \text{ locally
absolutely continuous in } [0,1]\setminus \{x_0\}, \\ & au \in
H^1_0(0,1), a u' \in H^1(0,1),  au \text{ is continuous at } x_0 \;\text{and}\\& (au)(x_0)=(au')(x_0)=u(0)=u(1)=0
\}.
\end{aligned}
\]
Then, under Hypothesis $\ref{Ass01}$ we have
\[
H^2_a(0,1)=D({\cal A})=D.
\]
\end{Proposition}
\begin{proof}
{$\boldsymbol{D\subseteq D({\cal A}):}$} Let $u \in D$. It is
sufficient to prove that $\sqrt{a}u' \in L^2(0,1)$. Since $au' \in
H^1(0,1)$ and $u(1)=0$ (recall that $a >0$ in $[0,1]\setminus
\{x_0\}$), for $x \in (x_0,1]$ we have
\[
\int_x^1[(au')'u](s)ds= [au'u]_x^1-\int_x^1(a(u')^2)(s)ds =
-(au'u)(x) -\int_x^1(a(u')^2)(s)ds.
\]
Thus
\[
(au'u)(x)=-\int_x^1[(au')'u](s)ds-\int_x^1(a(u')^2)(s)ds.
\]
Since $u \in D$, $(au')'u\in L^1(0,1)$. Hence, there exists
\[
\lim_{x \rightarrow x_0^+}(au'u)(x) =L\in[-\infty , +\infty),
\]
since no integrability is known about $a(u')^2$ and such a limit
could be $-\infty$. If $L\neq 0$, there exists $C>0$ such that
\[
|(au'u)(x)|\ge C
\]
for all $x$ in a right neighborhood of $x_0$, $x \neq x_0$. Thus, by
\cite[Lemma 2.5]{fggr}, there exists $C_1
>0$ such that
\[
|u(x)| \ge \frac{C}{|(au')(x)|} \ge\frac{C_1}{\sqrt{(x-x_0)}}
\]
for all $x$ in a right neighborhood of $x_0$, $x\neq x_0$. This
implies that $u \not\in L^2(0,1)$. Hence $L=0$ and
\begin{equation}\label{int}
\int_{x_0}^1[(au')'u](s)ds=-\int_{x_0}^1(a(u')^2)(s)ds.
\end{equation}
If $x\in [0, x_0)$, proceeding as before and using the condition
$u(0)=0$, it follows that
\begin{equation}\label{int1}
\int_0^{x_0}[(au')'u](s)ds=-\int_0^{x_0}(a(u')^2)(s)ds.
\end{equation}
By \eqref{int} and \eqref{int1}, it follows that
\[
\int_0^1[(au')'u](s)ds=-\int_0^1(a(u')^2)(s)ds.
\]
\noindent Since $(au')'u\in L^1(0,1)$, then $\sqrt{a}u' \in
L^2(0,1)$. Hence, $D \subseteq D({\cal A})$.
\\
{$\boldsymbol {D({\cal A}) \subseteq D:}$} Let $u \in D({\cal A})$.
By Proposition \ref{characterization}, we know that $au \in
H^1_0(0,1)$ and $(au)(x_0)=0$. Thus, it is sufficient to prove that
$(au')(x_0)=0$. Toward this end, observe that, since $au' \in
H^1(0,1)$, there exists $L\in \R$ such that $\displaystyle \lim_{x
\rightarrow x_0}(au')(x)=(au')(x_0)= L$. If $L \neq 0$, there exists
$C>0$ such that
\[
|(au')(x)| \ge C,
\]
for all $x$ in a neighborhood of $x_0$. Thus
\[
|(a(u')^2)(x)| \ge \frac{C^2}{a(x)},
\]
for all $x$ in a neighborhood of $x_0$, $ x \neq x_0$. By Lemma
\ref{rem}, this implies that $\sqrt{a}u' \not \in L^2(0,1)$. Hence
$L=0$, that is $(au')(x_0)=0$.
\end{proof}

Now, let us go back to problem \eqref{linear}, recalling the
following
\begin{Definition}\label{def_nondiv}
If $u_0 \in L^2(0,1)$ and $h\in L^2(Q_T)$, a function $u$ is said to
be a (weak) solution of \eqref{linear} if
\[
u \in C([0, T]; L^2(0,1)) \cap L^2(0, T; H^1_a(0,1))
\]
and
\[
\begin{aligned}
&\int_0^1u(T,x)\varphi(T,x)\, dx - \int_0^1 u_0(x) \varphi(0,x)\, dx
- \int_{Q_T}u\varphi_t \,dxdt =
\\&- \int_{Q_T} au_x
\varphi_x\,dxdt + \int_{Q_T} h\varphi \chi_\omega\,dx dt
\end{aligned}
\]
for all $\varphi \in H^1(0, T; L^2(0,1)) \cap L^2(0, T;
H^1_a(0,1))$.
\end{Definition}
As proved in \cite{fggr} (see Theorems $2.2$, $2.7$ and $4.1$), problem \eqref{linear} is well-posed in
the sense of the following theorem:
\begin{Theorem}\label{th-parabolic} For
all $h \in  L^2(Q_T)$ and $u_0 \in L^2(0,1)$, there exists a unique
weak solution  $u \in C([0,T]; L^2(0,1)) \cap L^2 (0,T; H^1_a(0,1))$
of \eqref{linear} and there exists a universal positive constant $C$
such that
\begin{equation}\label{stima}
\sup_{t \in [0,T]} \|u(t)\|^2_{L^2(0,1)}+\int_0^T\|u(t)\|^2_{H^1_a
(0,1)} dt \le C(\|u_0\|^2_{L^2(0,1)}+\|h\|^2_{L^2(Q_T)}).
\end{equation}
Moreover, if $u_0\in H^1_a(0,1)$, then
\begin{equation}\label{regularity}
u \in H^1(0,T; L^2(0,1))\cap C([0,T]; H^1_a(0,1)) \cap L^2(0,T;
H^2_a(0,1)),
\end{equation}
 and there exists a universal positive constant $C$ such
that
\begin{equation}\label{stima1}
\begin{aligned}
\sup_{t \in [0,T]}\left(\|u(t)\|^2_{H^1_a(0,1)} \right)&+ \int_0^{T}
\left(\left\|u_t\right\|^2_{L^2(0,1)} +
\left\|(au_x)_x\right\|^2_{L^2(0,1)}\right)dt\\&\le C
\left(\|u_0\|^2_{H^1_a(0,1)} + \|h\|^2_{L^2(Q_T)}\right).
\end{aligned}
\end{equation}
Moreover, $\mathcal A$ generates an analytic semigroup on
$L^2(0, 1)$.
\end{Theorem}

So far we have introduced all the tools which will let us deal with
solutions of problem \eqref{linear}, also with additional
regularity. Now, we conclude this section with an essential tool for
proving Carleman estimates and observability inequalities, that is a
new weighted Hardy--Poincar\'e inequality for functions which may
{\em not be} globally absolutely continuous in the domain, but whose
irregularity point is ``controlled'' by the fact that the weight
degenerates exactly there.
\begin{Proposition}[Hardy--Poincar\'{e} inequality]\label{HP}
Assume that $p \in C([0,1])$, $p>0$ on $[0,1]\setminus \{x_0\}$,
$p(x_0)=0$ and there exists $q \in (1,2)$ such that the function
$$
\begin{aligned}
x \longrightarrow \dfrac{p(x)}{|x-x_0|^{q}} &\mbox { is
nonincreasing on the left of } x=x_0 \\
& \mbox{ and nondecreasing on the right of } x=x_0.
\end{aligned}
$$
\noindent Then, there exists a constant $C_{HP}>0$ such that for any
function $w$, locally absolutely continuous on $[0,x_0)\cup (x_0,1]$
and satisfying
$$
w(0)=w(1)=0 \,\, \mbox{and } \int_0^1 p(x)|w^{\prime}(x)|^2 \,dx <
+\infty \,,
$$ the following inequality holds:
\begin{equation}\label{hardy1}
\int_0^1 \dfrac{p(x)}{(x-x_0)^2}w^2(x)\, dx \leq C_{HP}\, \int_0^1
p(x) |w^{\prime}(x)|^2 \,dx.
\end{equation}
\end{Proposition}
\begin{proof}
Fix any $\beta \in (1, q)$ and $\ve>0$ small. Since $w(1)=0$,
applying H\"{older}'s inequality and Fubini's Theorem, we have
\[
\begin{aligned}
&\int_{x_0+\ve}^1 \frac {p(x)}{(x-x_0)^2}w^2(x)\, dx \\
&= \int_{x_0+\ve}^1 \dfrac{p(x)}{(x-x_0)^2} \Big(\int_x^1
((y-x_0)^{\beta/2}w^{\prime}(y)) (y-x_0)^{-\beta/2}\,dy \Big)^2\,
dx \\
& \leq \int_{x_0+\ve}^1 \dfrac{p(x)}{(x-x_0)^2} \Big(\int_x^1
(y-x_0)^{\beta}|w^{\prime}(y)|^2 \,dy \int_x^1 (y-x_0)^{-\beta}\,dy
\Big)\, dx
\\
&\leq \dfrac{1}{\beta-1}\int_{x_0+\ve}^1
\dfrac{p(x)}{(x-x_0)^{1+\beta}} \Big(\int_x^1
(y-x_0)^{\beta}|w^{\prime}(y)|^2 \,dy \Big)\, dx
 \\
&= \dfrac{1}{\beta-1}\int_{x_0+\ve}^1
(y-x_0)^{\beta}|w^{\prime}(y)|^2
\Big(\int_{x_0+\ve}^y \dfrac{p(x)}{(x-x_0)^{1+\beta}} \,dx\Big)\, dy \\
&= \dfrac{1}{\beta-1}\int_{x_0+\ve}^1
(y-x_0)^{\beta}|w^{\prime}(y)|^2 \Big(\int_{x_0+\ve}^y
\dfrac{p(x)}{(x-x_0)^{q}} (x-x_0)^{q-1-\beta}\,dx\Big)\, dy.
\end{aligned}
\]
Now, thanks to our hypothesis, we find
\[
\dfrac{p(x)}{(x-x_0)^{q}} \le \dfrac{p(y)}{(y-x_0)^{q}}, \quad
\forall \;x,\; y \in \; [x_0+\ve, 1], \; x <y.
\]
Thus
\begin{equation}\label{HP1}
\begin{aligned}
&\int_{x_0+\ve}^1 \frac {p(x)}{(x-x_0)^2}w^2(x)\, dx \\
&\le \dfrac{1}{\beta-1}\int_{x_0+\ve}^1
\dfrac{p(y)}{(y-x_0)^{q}}(y-x_0)^{\beta}|w^{\prime}(y)|^2
\Big(\int_{x_0+\ve}^y (x-x_0)^{q-1-\beta}\,dx\Big)\, dy\\
&= \dfrac{1}{(\beta-1)(q-\beta)}\int_{x_0+\ve}^1
\dfrac{p(y)}{(y-x_0)^{q}}(y-x_0)^{q-\beta}(y-x_0)^{\beta}|w^{\prime}(y)|^2
dy\\
&= \dfrac{1}{(\beta-1)(q-\beta)}\int_{x_0+\ve}^1
p(y)|w^{\prime}(y)|^2 dy.
\end{aligned}
\end{equation}
Now, proceeding as before, and using the fact that $w(0)=0$, one has
\[
\begin{aligned}
&\int_0^{x_0-\ve}\frac {p(x)}{(x_0-x)^2}w^2(x)\, dx  \\& \leq
\int_0^{x_0-\ve} \dfrac{p(x)}{(x_0-x)^2} \Big(\int_0^x
(x_0-y)^{\beta}|w^{\prime}(y)|^2 \,dy \int_0^x (x_0-y)^{-\beta}\,dy
\Big)\, dx
\\
&\leq \dfrac{1}{\beta-1}\int_0^{x_0-\ve}
\dfrac{p(x)}{(x_0-x)^{1+\beta}} \Big(\int_0^x
(x_0-y)^{\beta}|w^{\prime}(y)|^2 \,dy \Big)\, dx
 \\
&= \dfrac{1}{\beta-1}\int_0^{x_0-\ve}
(x_0-y)^{\beta}|w^{\prime}(y)|^2 \Big(\int_y^{x_0-\ve}
\dfrac{p(x)}{(x_0-x)^{q}} (x_0-x)^{q-1-\beta}\,dx\Big)\, dy.
\end{aligned}
\]
By assumption
\[
\dfrac{p(x)}{(x_0-x)^{q}} \le \dfrac{p(y)}{(x_0-y)^{q}}, \quad
\forall \; x,\;y \in [0,x_0-\ve], \;y<x.
\]
Hence,
\begin{equation}\label{HP2}
\begin{aligned}
&\int_0^{x_0-\ve} \frac {p(x)}{(x_0-x)^2}w^2(x)\, dx \\
&\le \dfrac{1}{\beta-1}\int_0^{x_0-\ve}
\dfrac{p(y)}{(x_0-y)^{q}}(x_0-y)^{\beta}|w^{\prime}(y)|^2
\Big(\int_y^{x_0-\ve} (x_0-x)^{q-1-\beta}\,dx\Big)\, dy\\
&=
\dfrac{1}{(\beta-1)(q-\beta)}\int_0^{x_0-\ve}p(y)|w^{\prime}(y)|^2
dy.
\end{aligned}
\end{equation}
Passing to the limit as $\ve\to 0$ and combining \eqref{HP1} and
\eqref{HP2}, the conclusion follows.
\end{proof}

\section{Carleman Estimate for Degenerate Parabolic
Problems}\label{Carleman estimate}

In this section we prove a crucial estimate of Carleman type, that
will be useful to prove an observability inequality for the adjoint
problem of \eqref{linear} in both the weakly and the strongly
degenerate cases. Thus, let us consider the problem
\begin{equation}\label{1}
\begin{cases}
v_t + \left(av_x \right) _x =h, & (t,x) \in (0,T) \times (0,1),\\
v(t,0)=v(t,1)=0, &  t \in (0,T),\\
\end{cases}
\end{equation}
where $a$ satisfies the following assumption:
\begin{Assumptions} \label{Ass02} The function $a$ satisfies Hypothesis $\ref{Ass0}$ or Hypothesis $\ref{Ass01}$
and there exists $\vartheta \in (0, K]$ such that the function
\begin{equation}\label{ass02}
\begin{aligned}
x \longrightarrow \dfrac{a(x)}{|x-x_0|^{\vartheta}} &\mbox { is
nonincreasing on the left of } x=x_0 \\
& \mbox{ and nondecreasing on
the right of } x=x_0.
\end{aligned}
\end{equation}
Here $K$ is the constant appearing in Hypothesis $\ref{Ass0}$ or
$\ref{Ass01}$, respectively.
\end{Assumptions}

\begin{Remark}\label{domfurbo}
Observe that if $x_0=0$, Hypothesis \ref{Ass02} is the same
introduced in \cite{acf} only for the strongly degenerate case. On
the other hand, here we have to require this additional assumption
also in the weakly degenerate case. This is due to the fact that in
this case we don't know if $u(x_0)=0$ for all $u \in H^1_a(0,1)$, as
it happens when $x_0=0$ and one imposes homogeneous Dirichlet
boundary conditions, as in \cite{acf}; indeed, the choice of
homogeneous Dirichlet boundary conditions in \cite{acf} helps in
controlling the function at the degeneracy point, while here we
don't require the corresponding condition $u(x_0)=0$, so that some
other condition is needed. However, in both cases, Hypothesis
\ref{Ass02} is satisfied if $a(x)= |x-x_0|^K$, with $K \in (0,2)$.
\end{Remark}

Now, let us introduce the function $\varphi(t,x):
=\Theta(t)\psi(x)$, where
\begin{equation}\label{c_1}
\Theta(t) := \frac{1}{[t(T-t)]^4} \quad \text{and} \quad
%\]
%and
%\[
\psi(x) := c_1\left[\int_{x_0}^x \frac{y-x_0}{a(y)}dy- c_2\right],
\end{equation}
with $c_2> \displaystyle \max\left\{\frac{(1-x_0)^2}{a(1)(2-K)},
\frac{x_0^2}{a(0)(2-K)}\right\}$ and $c_1>0$. A more precise
restriction on $c_1$ will be needed later. Observe that $\Theta (t)
\rightarrow + \infty \, \text{ as } t \rightarrow 0^+, T^-$, and by
Lemma \ref{rem} we have that, if $x>x_0$,
\begin{equation}\label{defpsi}
\begin{aligned}
\psi(x)&\leq c_1 \left[ \int_{x_0}^x
\frac{(y-x_0)^K}{a(y)}\frac{1}{(y-x_0)^{K-1}}dy- c_2\right]\\
&\leq c_1\left[
\frac{(x-x_0)^K}{a(x)}\frac{(x-x_0)^{2-K}}{2-K}-c_2\right]\leq c_1
\left[
\frac{(1-x_0)^K}{a(1)}\frac{(1-x_0)^{2-K}}{2-K}-c_2\right]\\
&=c_1\left[ \frac{(1-x_0)^2}{(2-K)a(1)}-c_2\right]<0.
\end{aligned}
\end{equation}
In the same way one can treat the case $x\in[0,x_0)$, so that
\[
\psi(x)<0 \quad \mbox{ for every }x\in[0,1].
\]
Moreover, it is also easy to see that $\psi \ge -c_1c_2$.

\medskip
Our main result is thus the following.
\begin{Theorem}\label{Cor1}
Assume Hypothesis  $\ref{Ass02}$ and let $T>0$. Then, there exist
two positive constants $C$ and $s_0$, such that every solution $v$
of \eqref{1} in
\begin{equation}\label{v}
\mathcal{V}:=L^2\big(0, T; H^2_a(0,1)\big) \cap H^1\big(0,
T;H^1_a(0,1)\big)
\end{equation}
satisfies, for all $s \ge s_0$,
\[
\begin{aligned}
&\int_0^T\int_0^1 \left(s\Theta a(v_x)^2 + s^3 \Theta^3
\frac{(x-x_0)^2}{a} v^2\right)e^{2s\varphi}dxdt\\
&\le C\left(\int_0^T\int_0^1 |h|^{2}e^{2s\varphi}dxdt +
sc_1\int_0^T\left[a\Theta e^{2s \varphi}(x-x_0)(v_x)^2
dt\right]_{x=0}^{x=1}\right),
\end{aligned}
\]
where $c_{1}$ is the constant introduced in \eqref{c_1}.
\end{Theorem}
\subsection{Proof of Theorem \ref{Cor1}}

%\begin{Remark}{\rm The inequality proved in the previous proposition still holds under weaker hypotheses on the function
%$a$. In particular, it holds if we substitute \eqref{ass02} and
%\eqref{ass03} with the following ones
%\[
%\begin{aligned}
%(3.2') \quad &\exists\; \vartheta \in (0, K)\mbox { such that the
%function} \;x \longrightarrow \dfrac{a(x)}{|x-x_0|^{\vartheta}} \mbox {
%is nonincreasing on a left neighbourhood of } \\& x=x_0  \mbox{ and
%nondecreasing on a right neighbourhood of } x=x_0.
%\end{aligned}
%\]
%or
%\[
%\begin{aligned}
%(3.3') \quad &\exists\; \vartheta \in (1,K), \mbox{ if } K
%>1, \mbox {or } \vartheta \in (0,1), \mbox {if } K=1, \mbox{ such that the
%function } \; x \longrightarrow \dfrac{a(x)}{|x-x_0|^{\vartheta}} \mbox
%{ is nonincreasing}\\& \mbox{on a left neighbourhood of } x=x_0
%\mbox{ and nondecreasing on a right neighbourhood of } x=x_0,
%\end{aligned}
%\] respectively. }
%\end{Remark}
For $s> 0$, define the function
\[
w(t,x) := e^{s \varphi (t,x)}v(t,x),
\]
where $v$ is any solution of \eqref{1} in $\mathcal{V}$; observe
that, since $v\in\mathcal{V}$ and $\psi<0$, then $w\in\mathcal{V}$.
Of course, $w$ satisfies
\begin{equation}\label{1'}
\begin{cases}
(e^{-s\varphi}w)_t + \left(a(e^{-s\varphi}w)_x \right) _x  =h, & (t,x) \in (0,T) \times (0,1),\\
w(t,0)=w(t,1)=0, &  t \in (0,T),\\ w(T^-,x)= w(0^+, x)= 0, & x \in
(0,1).
\end{cases}
\end{equation}
The previous problem can be recast as follows. Set
\[
Lv:= v_t + (av_x)_x \quad \text{and} \quad
%\]
%and
%\[
L_sw= e^{s\varphi}L(e^{-s\varphi}w), \quad s  > 0.
\]
Then \eqref{1'} becomes
\[
\begin{cases}
L_sw= e^{s\varphi}h,\\
w(t,0)=w(t,1)=0, & t \in (0,T),\\
w(T^-,x)= w(0^+, x)= 0, & x \in (0,1).
\end{cases}
\]
Computing $L_sw$, one has
\[
\begin{aligned}
L_sw
%& = (\partial_t - s\varphi_t)w + (\partial_x - s
%\varphi_x)[a(x)(\partial_x - s \varphi_x)w] \\
%&= (a(x)w_x)_x
% - s \varphi_t w + s^2a(x) \varphi_x^2 w + w_t -2sa(x)\varphi_x w_x -
% s(a(x)\varphi_x)_xw\\
% &
=L^+_sw + L^-_sw,
\end{aligned}
\]
where
\[
L^+_sw := (aw_x)_x
 - s \varphi_t w + s^2a (\varphi_x)^2 w,
\]
and
\[
L^-_sw := w_t -2sa\varphi_x w_x -
 s(a\varphi_x)_xw.
\]
Moreover,
\begin{equation}\label{stimetta}
\begin{aligned}
2\langle L^+_sw, L^-_sw\rangle &\le 2\langle L^+_sw, L^-_sw\rangle+
\|L^+_sw \|_{L^2(Q_T)}^2 + \|L^-_sw\|_{L^2(Q_T)}^2\\
& =\| L_sw\|_{L^2(Q_T)}^2= \|he^{s\varphi}\|_{L^2(Q_T)}^2,
\end{aligned}
\end{equation}
where $\langle\cdot, \cdot \rangle$ denotes the
usual scalar product in $L^2(Q_T)$. As usual, we will separate the
scalar product $\langle L^+_sw, L^-_sw\rangle$ in distributed terms
and boundary terms.
\begin{Lemma}\label{lemma1}
The following identity holds:
\begin{equation}\label{D&BT}
\left.
\begin{aligned}
&\langle L^+_sw,L^-_sw\rangle \;\\&=\; \frac{s}{2} \int_0^T \int_0^1
\varphi_{tt} w^2dxdt+ s^3 \int_0^T \int_0^1(2a \varphi_{xx} +
a'\varphi_x)a (\varphi_x)^2 w^2dxdt
\\&- 2s^2 \int_0^T \int_0^1a \varphi_x \varphi_{tx}w^2dxdt  +s
\int_0^T \int_0^1(2 a^2\varphi_{xx} + aa'\varphi_x)(w_x)^2 dxdt
\end{aligned}\right\}\;\text{\{D.T.\}}
\end{equation}
%\vspace{-5pt}
\begin{equation}\nonumber
 \text{\{B.T.\}}\;\left\{
\begin{aligned}
& + \int_0^T[aw_xw_t]_{x=0}^{x=1} dt- \frac{s}{2}
\int_0^1[w^2\varphi_t]_{t=0}^{t=T}dx+ \frac{s^2}{2}\int_0^1
[a(\varphi_x)^2 w^2]_{t=0}^{t=T}dt\\& + \int_0^T[-s\varphi_x
(aw_x)^2 +s^2a\varphi_t \varphi_x w^2 - s^3 a^2(\varphi_x)^3w^2
]_{x=0}^{x=1}dt
\\&+ \int_0^T[-sa(a\varphi_x)_xw
w_x]_{x=0}^{x=1}dt-\frac{1}{2} \int_0^1 \Big[a(w_x)^2\Big]_0^Tdx.
\end{aligned}\right.
\end{equation}
\end{Lemma}\begin{proof}
First, let us note that all integrals which appear in $\langle
L^+_sw,L^-_sw\rangle$ are well defined both in the weakly and the
strongly degenerate case by Theorem \ref{rem}, as simple calculations
show, recalling that $w=e^{s\vp}v$. Moreover, we remark that all the
following integrations by parts are justified by Proposition
\ref{domain} and the fact that $w\in {\mathcal V}$. Hence
\begin{equation}\label{!}
\begin{aligned}
&\int_0^T \int_0^1L^+_sw w_t  dxdt= \int_0^T \int_0^1 \{ (aw_x)_x
 - s \varphi_t w + s^2a(\varphi_x)^2 w\}w_t dxdt\\
 &=\int_0^T[aw_xw_t]_{x=0}^{x=1}dt -
 \int_0^T\frac{1}{2}\frac{d}{dt}\left(\int_0^1a(w_x)^2dx\right)dt\\
 & - \frac{s}{2}\int_0^Tdt \int_0^1\varphi_t(w^2)_tdx +
 \frac{s^2}{2}\int_0^Tdt \int_0^1a(\varphi_x)^2(w^2)_tdx
\\
&= \int_0^T[aw_xw_t]_{x=0}^{x=1}dt - \frac{s}{2}
\int_0^1[w^2\varphi_t]_{t=0}^{t=T}dx+ \frac{s^2}{2}\int_0^1
[a(\varphi_x)^2 w^2]_{t=0}^{t=T}dt\\&
 -\int_0^T\frac{1}{2}\frac{d}{dt}\left(\int_0^1a(w_x)^2dx\right)dt+ \frac{s}{2}\int_0^T
\int_0^1\varphi_{tt}w^2dxdt\\& - s^2\int_0^T \int_0^1a \varphi_x
\varphi_{xt}w^2 dxdt.
\end{aligned}
\end{equation}

In addition, we have
\begin{equation}\label{!1}
\begin{aligned}
&\int_0^T \int_0^1L^+_sw (-2sa\varphi_xw_x)  dxdt= -2s\int_0^T
\int_0^1\varphi_x\left[\frac{(aw_x)^2}{2} \right]_x  dxdt \\&+
2s^2\int_0^T \int_0^1a \varphi_t \varphi_x \left(
\frac{w^2}{2}\right)_x
dxdt-2s^3 \int_0^T \int_0^1a^2(\varphi_x)^3 w w_x  dxdt\\
 &=\int_0^T[-s\varphi_x(aw_x)^2 + s^2a\varphi_t\varphi_xw^2 - s^3a^2(\varphi_x)^3w^2]_{x=0}^{x=1}dt
 \\&+
 s\int_0^T\int_0^1 \varphi_{xx} (aw_x)^2dxdt  -s^2 \int_0^T\int_0^1 (a\varphi_t \varphi_x)_x w^2  dxdt
 \\&+ s^3\int_0^T\int_0^1\{(a(\varphi_x)^2)_x a \varphi_x  + a (\varphi_x)^2 (a
 \varphi_x)_x\}w^2 dxdt.
\end{aligned}
\end{equation}
At this point, note that $(a \varphi_x)_x=c_1\Theta$, so that $(a
\varphi_x)_{xx}=0$.

Moreover
\begin{equation}\label{!2}
\begin{aligned}
&\int_0^T \int_0^1L^+_sw (-s(a\varphi_x)_xw) dxdt =\int_0^T[-saw_x
w(a \varphi_x)_x ]_{x=0}^{x=1}dt\\
&+s\int_0^T\int_0^1 a w_x (a \varphi_x)_x w_x dxdt \\
 & + s^2\int_0^T\int_0^1(a\varphi_x)_x \varphi_t w^2  dxdt- s^3
 \int_0^T \int_0^1 a (\varphi_x)^2 (a \varphi_x)_x w^2 dxdt.
\end{aligned}
\end{equation}
    Adding \eqref{!}-\eqref{!2}, \eqref{D&BT} follows immediately.
\end{proof}

Now, the crucial step is to prove the following estimate:
\begin{Lemma}\label{lemma2}
Assume Hypothesis $\ref{Ass02}$. Then there exists a positive
constant $s_0$ such that for all $s \ge s_{0}$ the distributed terms
of \eqref{D&BT} satisfy the estimate
\[
\begin{aligned}
&\frac{s}{2}\int_0^T \int_0^1 \varphi_{tt} w^2 dxdt + s^3 \int_0^T
\int_0^1(2a \varphi_{xx}+ a' \varphi_x)a (\varphi^2_x) w^2
dxdt\\
&- 2s^2 \int_0^T \int_0^1a \varphi_x \varphi_{tx}w^2 dxdt + s
\int_0^T \int_0^1(2 a^2\varphi_{xx} + aa' \varphi_x)(w_x)^2 dxdt \\
&\ge \frac{C}{2}s\int_0^T\int_0^1 \Theta a(w_x)^2 dxdt +
\frac{C^3}{2}s^3 \int_0^T\int_0^1\Theta^3 \frac{(x-x_0)^2}{a} w^2
dxdt,
\end{aligned}
\]
for a positive constant $C$.
\end{Lemma}
\begin{proof}
Using the definition of $\varphi$, the distributed terms of
$\int_0^T \int_0^1L^+_s w L^-_s w dxdt$ take the form
\begin{equation}\label{02}
\begin{aligned}
%&\frac{s}{2} \int_0^T \int_0^1 \varphi_{tt} w^2  dxdt+ s\int_0^T
%\int_0^1a(x) (a(x) \varphi_x)_{xx} w w_x  dxdt- 2s^2 \int_0^T
%\int_0^1a(x) \varphi_x \varphi_{tx}w^2 dxdt \\&+ s\int_0^T \int_0^1(2
%a^2\varphi_{xx} + a(x)a'\varphi_x)(w_x)^2 dxdt+ s^3 \int_0^T
%\int_0^1(2a(x) \varphi_{xx} + a'\varphi_x)a \varphi^2_x
%w^2 dxdt\\
%& =
&\frac{s}{2} \int_0^T \int_0^1 \ddot{\Theta} \psi w^2 dxdt +
s^3\int_0^T \int_0^1\Theta^3a(2a\psi''+ a' \psi')(\psi')^2w^2
dxdt\\
& -2s^2\int_0^T \int_0^1\Theta\dot{\Theta}a(\psi')^2w^2 dxd+ s
\int_0^T \int_0^1\Theta a(2a\psi''+ a' \psi')(w_x)^2
dxdt.
\end{aligned}
\end{equation}
Because of the choice of $\psi(x)$, one has
\[
2a(x) \psi''(x) + a'(x) \psi' (x)= c_1\frac{2a(x)-a'(x)
(x-x_0)}{a(x)}.
\]
Thus \eqref{02} becomes
\[
\begin{aligned}
\frac{s}{2}\int_0^T \int_0^1\ddot{\Theta}\psi w^2  dxdt&-2s^2
\int_0^T \int_0^1\Theta \dot{\Theta}a(\psi')^2w^2 dxdt\\
&+ sc_1 \int_0^T \int_0^1\Theta  (2a(x) -a'(x) (x-x_0))(w_x)^2
dxdt\\& + s^3c_1 \int_0^T \int_0^1\Theta^3(\psi')^2(2a(x) -a'(x)
(x-x_0))w^2 dxdt.
\end{aligned}
\]
By assumption, one can estimate the previous terms in the following
way:
\[
\begin{aligned}
&\frac{s}{2}\int_0^T \int_0^1 \ddot{\Theta}\psi w^2  dxdt-2s^2
\int_0^T \int_0^1\Theta \dot{\Theta}a(\psi')^2w^2 dxdt\\
& + sc_1 \int_0^T \int_0^1\Theta (2a(x) -a'(x)(x-x_0))(w_x)^2
dxdt\\& + s^3c_1\int_0^T \int_0^1\Theta^3 (\psi')^2(2a(x)
-a'(x)(x-x_0))w^2 dxdt
\\&\ge -2s^2 \int_0^T \int_0^1\Theta
\dot{\Theta}a(\psi')^2w^2  dxdt +\frac{s}{2}\int_0^T \int_0^1
\ddot{\Theta}\psi w^2dxdt \\&+ sC \int_0^T \int_0^1\Theta a(w_x)^2
dxdt+ s^3C^3\int_0^T \int_0^1\Theta^3 \frac{(x-x_0)^2}{a}w^2 dxdt,
\end{aligned}
\]
where $C>0$ is some universal positive constant.
 Observing that
%\begin{equation}\label{theta}
 $|\Theta \dot{\Theta}| \le c \Theta^{9/4} \le c \Theta ^3$ and $|\ddot{\Theta}| \le c\Theta
^{3/2}$, for a
 positive constant $c$,
%\end{equation}
we conclude that, for $s$ large enough,
\[
\begin{aligned}
&\left|-2s^2\int_0^T \int_0^1\Theta  \dot{\Theta} a(\psi')^2w^2 dxdt
\right| \le 2c s^2
\int_0^T \int_0^1\Theta^3a(\psi')^2w^2 dxdt\\
&= 2c s^2\int_0^T \int_0^1\Theta^3 c_1^2\frac{(x-x_0)^2}{a}w^2 dxdt
\le \frac{C^3}{4}s^3\int_0^T \int_0^1\Theta^3 \frac{(x-
x_0)^2}{a}w^2 dxdt.
\end{aligned}
\]
Moreover,
\begin{equation}\label{zio}
\begin{aligned}
\left|  \frac{s}{2} \int_0^T \int_0^1 \ddot{\Theta}\psi w^2dxdt
\right| &\le \frac{s}{2}c_1c \left|\int_0^T \int_0^1\Theta^{3/2}
bw^2 dxdt\right| \\&+ s\frac{c_1c_2}{2}c\left|\int_0^T
\int_0^1\Theta ^{3/2}w^2 dxdt\right|,
\end{aligned}
\end{equation}
where $b(x)= \displaystyle\int_{x_0}^x\frac{y-x_0}{a(y)}dy\geq0$.
Now, since the function $x \mapsto\displaystyle
\frac{|x-x_0|^K}{a(x)}$ is nonincreasing on $[0, x_0)$ and
nondecreasing on $(x_0,1]$ (see Lemma \ref{rem}), one has $ b(x)\le
\displaystyle \frac{(x-x_0)^2}{(2-K)a(x)}$, see \eqref{defpsi}.
Hence
\[
\frac{s}{2}c_1c\int_0^T \int_0^1\Theta^{3/2} bw^2 dxdt
\le\frac{C^3}{8}s^3\int_0^T \int_0^1\Theta^{3} \frac{(x-
x_0)^2}{a}w^2 dxdt,
\]
for $s$ large enough.

It remains to bound the term $\left|\int_0^T \int_0^1\Theta
^{3/2}w^2 dxdt\right|$. Using the Young inequality, we find
\begin{equation}\label{sopra}
\begin{aligned} &s\frac{c_1c_2}{2}c\left|\int_0^1
\Theta^{3/2}w^2 dx \right|\\&=
s\frac{c_1c_2}{2}c\left|\int_0^1\left(\Theta\frac{a^{1/3}}{|x-x_0|^{2/3}}w^2\right)^{3/4}\left(\Theta^3
\frac{|x-x_0|^2}{a} w^2\right)^{1/4}\right| \\
&\le s\frac{3 c_1c_2}{4}c \int_0^1
\Theta\frac{a^{1/3}}{|x-x_0|^{2/3}}w^2dx +
 s\frac{c_1c_2}{4}c \int_0^1\Theta^3
\frac{|x-x_0|^2}{a} w^2 dx.
\end{aligned}
\end{equation}

Now, consider the function $p(x) = (a(x)|x-x_0|^4)^{1/3}$. It is
clear that, setting  $C_1:=
\max\left\{\displaystyle\left(\frac{x_0^2}{a(0)}\right)^{2/3},\displaystyle\left(\frac{(1-x_0)^2}{a(1)}\right)^{2/3}\right\}$,
by Lemma \ref{rem} we have $\displaystyle p(x)=  a(x)
\left(\frac{(x-x_0)^2}{a(x)}\right)^{2/3}\le C_1 a(x)$ and
$\displaystyle \frac{a^{1/3}}{|x-x_0|^{2/3}}=
\frac{p(x)}{(x-x_0)^2}$. Moreover, using Hypothesis \ref{Ass02}, one
has that the function $\displaystyle\frac{p(x)}{|x-x_0|^q}$, where
$\displaystyle q: =\frac{4+\vartheta}{3}\in(1,2)$, is nonincreasing
on the left of $x=x_0$ and nondecreasing on the right of $x=x_0$.
The Hardy-Poincar\'{e} inequality (see Proposition \ref{HP}) implies
\begin{equation}\label{hpappl}
\begin{aligned}
\int_0^1 \Theta\frac{a^{1/3}}{|x-x_0|^{2/3}}w^2dx &= \int_0^1 \Theta
\frac{p}{(x-x_0)^2} w^2 dx \le C_{HP}\int_0^1 \Theta p (w_x)^2 dx
\\&\le C_{HP}C_1 \int_0^1 \Theta a (w_x)^2 dx,
\end{aligned}
\end{equation}
where $C_{HP}$ and $C_1$ are the Hardy-Poincar\'{e} constant and the
constant introduced before, respectively. Thus, for $s$ large
enough, by \eqref{sopra} and \eqref{hpappl}, we have
\[
\begin{aligned}
s\frac{c_1c_2}{2}c \int_0^1 \Theta^{3/2} w^2 dx \le
\frac{C}{2}s\int_0^1 \Theta a (w_x)^2 dx + \frac{C^3}{8}s^3
\int_0^1\Theta^3 \frac{(x-x_0)^2}{a}w^2 dxdt,
\end{aligned}
\]
for a positive constant $C$. Using the estimates above, from
\eqref{zio} we finally obtain
 \[
 \begin{aligned}
\left|  \frac{s}{2} \int_0^T \int_0^1 \ddot{\Theta}\psi w^2dxdt
\right| &\le \frac{C}{2}s\int_0^T\int_0^1 \Theta a (w_x)^2 dxdt
\\&+ \frac{C^3}{4}s^3\int_0^T \int_0^1\Theta^3
\frac{(x-x_0)^2}{a}w^2 dxdt.
\end{aligned}
\]
Summing up, we obtain
\[
\begin{aligned}
&\frac{s}{2} \int_0^T \int_0^1 \varphi_{tt} w^2  dxdt+ s \int_0^T
\int_0^1a (a \varphi_x)_{xx} w w_x dxdt \\&- 2s^2 \int_0^T \int_0^1a
\varphi_x \varphi_{tx}w^2  dxdt+ s \int_0^T \int_0^1(2
a^2\varphi_{xx} + aa'\varphi_x)(w_x)^2  dxdt\\&+ s^3 \int_0^T
\int_0^1(2a \varphi_{xx}+ a'\varphi_x)a (\varphi_x)^2 w^2 dxdt\\&\ge
\frac{C}{2}s\int_0^T\int_0^1 \Theta a(w_x)^2 dxdt + \frac{C^3}{2}s^3
\int_0^T\int_0^1\Theta^3 \frac{(x-x_0)^2}{a} w^2 dxdt.
\end{aligned}
\]
\end{proof}
For the boundary terms in \eqref{D&BT}, it holds:
\begin{Lemma}\label{lemma4}
The boundary terms in \eqref{D&BT} reduce to
\[-s \int_0^{T} \left[ \Theta
     (aw_{x})^2 \psi' \right]_{x=0}^{x=1}dt.
\]
\end{Lemma}
\begin{proof}
Using the definition of $\varphi$, we have that the boundary terms
become
\begin{equation}\label{bt}
\begin{aligned}
\text{ (B.T.) } =  &\int_0^{T} [a w_{x}w_{t} -s a\Theta(a\psi')'ww_x
+ s^{2} \dot{\Theta}\Theta a \psi \psi' w^{2}\\&\quad\quad- s^{3}
a^{2}\Theta^{3}(\psi')^3w^2- s\Theta (aw_{x})^2 \psi'
]_{x=0}^{x=1}dt\\&+\int_0^1\left[-\frac{s}{2}w^2\psi
\dot{\Theta}+\frac{s^2}{2}aw^2(\psi')^2\Theta^2-\frac{1}{2}
a(w_x)^2\right]_{t=0}^{t=T} dx.
\end{aligned}
\end{equation}
Since $w \in \mathcal{V}$, $ w\in C\big([0, T];H^1_a(0,1) \big)$.
Thus $w(0, x)$, $w(T,x)$, $w_x(0,x)$, $w_x(T,x)$ and $\int_0^1
\big[a(w_x)^2\big]_0^Tdx$ are indeed well defined. Using the
boundary conditions of $w$ and the definition of $w$, we get that
\[
\int_0^1\left[-\frac{s}{2}w^2\psi
\dot{\Theta}+\frac{s^2}{2}aw^2(\psi')^2\Theta^2-\frac{1}{2}a(w_x)^2\right]_{t=0}^{t=T}
dx=0.
\]

Moreover, since $w \in \mathcal{V}$, we have that
 $w_t (t,0)$ and $w_t(t,1)$ make sense. Therefore, also $a(0)w_x(t,0)$ and $a(1)w_x(t,1)$
are well defined. In fact  $w(t,\cdot)\in H^2_a(0,1)$ and
$a(\cdot)w_x(t,\cdot)\in W^{1,2}(0,1)\subset C([0,1])$.
%Therefore, $w \in H^1(0,T; H^1_a(0,1))$ implies
%that $w_t \in L^2(0,T;H^1_a(0,1))$ and, in particular,
%$\sqrt{a}w_{tx} (t,x) \in L^2(0,1)$. Then, by H\"{o}lder's
%inequality and using the fact that $w_t(t,1)=0$, for any $t \in (0,
%T)$ and $ x \in (0,1)$ we have
%\[
%|w_t(t,x)|\le \int_x^{1} |w_{tx}(t,y)| dy \le
%\left\|\frac{1}{a}\right\|_{L^1(0,1)}\left(\int_x^{1}
%a|w_{tx}(t,y)|^2dy\right)^{1/2}.
%\]
Thus $ \int_0^T[a w_xw_t]_{x=0}^{x=1}dt$ is well defined and
actually equals $0$, as we get using the  boundary conditions on
$w$.

Now, consider the second, the third and the fourth terms of
\eqref{bt}. By definition of $\psi$ and using the hypothesis on $a$,
the functions $(a\psi')'$, $a\psi\psi'$ and $a^2(\psi')^3$ are
bounded on $[0,1]$. Thus, by the boundary conditions on $w$, one has
\[
\begin{aligned}
s\int_0^T \left[  a\Theta(a\psi')'ww_x\right]_{x=0}^{x=1}dt& =
s^{2}\int_0^T \left[ \dot{\Theta}\Theta a \psi
     \psi' w^{2}\right]_{x=0}^{x=1}dt\\&=
s^{3}\int_0^T \left[
a^{2}\Theta^{3}(\psi')^3w^2\right]_{x=0}^{x=1}dt = 0.
\end{aligned}
\]
\end{proof}

From Lemma \ref{lemma1}, Lemma \ref{lemma2}, and Lemma \ref{lemma4},
we deduce immediately that there exist two positive constants $C$
and $s_0$, such that all solutions $w$ of \eqref{1'} satisfy, for
all $s \ge s_0$,
\begin{equation}\label{D&BT1}
\begin{aligned}
\int_0^T \int_0^1L^+_s w L^-_s w dxdt &\ge Cs\int_0^T\int_0^1 \Theta
a(w_x)^2 dxdt\\&+ Cs^3 \int_0^T\int_0^1\Theta^3 \frac{(x-x_0)^2}{a}
w^2 dxdt\\&- s\int_0^{T} \left[\Theta
a^2(w_{x})^{2}\psi'\right]_{x=0}^{x=1}dt.
\end{aligned}
\end{equation}

Thus, a straightforward consequence of \eqref{stimetta} and of
\eqref{D&BT1} is the next result.
\begin{Proposition}\label{Carleman}
Assume Hypothesis $\ref{Ass02}$ and let $T>0$. Then, there exist two
positive constants $C$ and $s_0$, such that all solutions $w$ of
\eqref{1'} in $\mathcal{V}$ satisfy, for all $s \ge s_0$, \[
\begin{aligned}
   & s\int_0^T\int_0^1 \Theta a(w_x)^2dxdt  + s^3
\int_0^T\int_0^1\Theta^3 \frac{(x-x_0)^2}{a} w^2 dxdt\\&\le
    C\left(\int_0^T\int_0^1 |h|^2 e^{2s\varphi(t,x)}dxdt+ s\int_0^{T} \left[\Theta
       a^2(w_{x})^{2}\psi'\right]_{x=0}^{x=1}dt.
    \right).
    \end{aligned}
    \]
\end{Proposition}
Recalling the definition of $w$, we have $v= e^{-s\varphi}w$ and
$v_{x}= -s\Theta \psi'e^{-s\varphi}w + e^{-s\varphi}w_{x}$. Thus, by
the Cauchy--Schwarz inequality
\[
\begin{aligned}
&\int_0^T\int_0^1 \left(s\Theta a(v_x)^2 + s^3 \Theta^3
\frac{(x-x_0)^2}{a} v^2\right)e^{2s\varphi}dxdt\\
&\leq \int_0^T\int_0^1 \left(s\Theta a (w_x)^2dxdt+s^3c_1^2\Theta^3
\frac{(x-x_0)^2}{a}w^2\right)dxdt,
\end{aligned}
\]
and by Proposition \ref{Carleman}, Theorem \ref{Cor1} follows.

\section{Application of Carleman estimates to observability inequalities}\label{sec4}

In this section we provide a possible application of the Carleman
estimates established in the previous section, considering the
control problem \eqref{linear}. In particular, we consider the
situation in which $x_0$ is inside the control interval
\begin{equation}\label{omega1}
x_0\in\omega=(\alpha,\beta) \subset (0,1).
\end{equation}

Now, we associate to the linear problem \eqref{linear} the
homogeneous adjoint problem
\begin{equation}\label{h=0}
\begin{cases}
v_t +(av_x)_x= 0, &(t,x) \in  Q_T,
\\[5pt]
v(t,0)=v(t,1) =0, & t \in (0,T),
\\[5pt]
v(T,x)= v_T(x)\in L^2(0,1),
\end{cases}
\end{equation}
where $T>0$ is given.
By the Carleman estimate in Theorem
\ref{Cor1}, we will deduce the following observability inequality
for both the weakly and the strongly degenerate cases:
\begin{Proposition}\label{obser.}
Assume Hypothesis $\ref{Ass02}$ and \eqref{omega1}.
Then there exists a positive constant $C_T$ such that every solution
$v \in  C([0, T]; L^2(0,1)) \cap L^2 (0,T; H^1_a(0,1))$ of
\eqref{h=0} satisfies
 \begin{equation}\label{obser1.}
\int_0^1v^2(0,x) dx \le C_T\int_0^T \int_{w}v^2(t,x)dxdt.
\end{equation}
\end{Proposition}

\subsection{Proof of Proposition \ref{obser.}} In
this subsection we will prove, as a consequence of the Carleman
estimate proved in Section \ref{Carleman estimate}, the
observability inequality \eqref{obser1.}. For this purpose, we will
give some preliminary results. As a first step, we consider the
adjoint problem with more regular final--time datum
\begin{equation}\label{h=01}
\begin{cases}
v_t +(av_x)_{x}= 0, &(t,x) \in  Q_T,
\\[5pt]
v(t,0)=v(t,1) =0, & t \in (0,T),
\\[5pt]
v(T,x)= v_T(x) \,\in D({\cal A}^2),
\end{cases}
\end{equation}
where
\[
D({\cal A}^2) = \Big\{u \,\in \,D({\cal A})\;\big|\; {\cal A}u \,\in
\,D({\cal A}) \;\Big\}
\]
and ${\cal A}u:=(au_x)_x$. Observe that $D({\cal A}^2)$ is densely
defined in $D({\cal A})$ (see, for example, \cite[Lemma 7.2]{b}) and
hence in $L^2(0,1)$. As in \cite{cfr}, \cite{cfr1} or \cite{f},
letting $v_T$ vary in $D({\cal A}^2)$, we define the following class
of functions:
\[
\cal{W}:=\Big\{ v\text{ is a solution of \eqref{h=01}}\Big\}.
\]
Obviously (see, for example, \cite[Theorem 7.5]{b})
\[ \cal{W}\subset
C^1\big([0,T]\:;\:H^2_a(0,1)\big) \subset \mathcal{V} \subset
\cal{U},
\]
where, $\mathcal{V}$ is defined in \eqref{v} and
\begin{equation}\label{U}
\cal{U}:= C([0,T]; L^2(0,1)) \cap L^2(0, T; H^1_a(0,1)).
\end{equation}

We start with
\begin{Proposition}[Caccioppoli's inequality]\label{caccio}
Let $\omega'$ and $\omega$ two open subintervals of $(0,1)$ such
that $\omega'\subset \subset \omega \subset  (0,1)$ and $x_0 \not
\in \bar\omega'$. Let $\varphi(t,x)=\Theta(t)\Upsilon(x)$, where
$\Theta$ is defined in \eqref{c_1} and
\[
\Upsilon \in C([0,1],(-\infty,0))\cap
C^1([0,1]\setminus\{x_0\},(-\infty,0))
\]
is such that
\begin{equation}\label{stimayx}
|\Upsilon_x|\leq \frac{c}{\sqrt{a}} \mbox{ in }[0,1]\setminus\{x_0\}
\end{equation}
for some $c>0$. Then, there exist two positive constants $C$ and
$s_0$ such that every solution $v \in \cal W$ of the adjoint problem
\eqref{h=01} satisfies
\begin{equation}\label{lemme-caccio}
   \int_{0}^T \int _{\omega'}   (v_x)^2e^{2s\varphi } dxdt
    \ \leq \ C \int_{0}^T \int _{\omega}   v^2  dxdt,
\end{equation}
for all $s\geq s_0$.
\end{Proposition}

\begin{Remark}
Of course, our prototype for $\Upsilon$ is the function $\psi$
defined in \eqref{c_1}. Indeed,
\[
|\psi'(x)|=c_1\left|\frac{x-x_0}{a(x)}\right|=c_1\sqrt{\frac{|x-x_0|^2}{a(x)}}\frac{1}{\sqrt{a(x)}}\leq
c\frac{1}{\sqrt{a(x)}}
\]
by Lemma \ref{rem}.
\end{Remark}

\begin{proof}[Proof of Proposition $\ref{caccio}$]
Let us consider a smooth function $\xi: [0,1] \to \Bbb R$ such that
  \[\begin{cases}
    0 \leq \xi (x)  \leq 1, &  \text{for all } x \in [0,1], \\
    \xi (x) = 1 ,  &   x \in \omega', \\
    \xi (x)=0, &     x \in [\, 0, 1 ]\setminus \omega.
    \end{cases}\]
Since $v$ solves \eqref{h=01} and has homogeneous boundary
conditions, by the choice of $\vp$, we have
\[
    \begin{aligned}
    0 &= \int _0 ^T \frac{d}{dt} \left(\int _0 ^1 \xi ^2 e^{2s\varphi}
    v^2dx\right)dt
    =  \int_0^T \int_0^1(2s \xi ^2  \varphi _t e^{2s\varphi} v^2 + 2 \xi ^2
    e^{2s\varphi} vv_t )dxdt
    \\
    &= 2 s\int_0^T \int_0^1 \xi ^2 \varphi _t e^{2s\varphi} v^2dxdt + 2 \int_0^T \int_0^1\xi
    ^2e^{2s\varphi} v (-(a v_x)_x)dxdt
    \\
    &= 2s \int_0^T \int_0^1\xi ^2 \varphi _t e^{2s\varphi} v^2 dxdt +  2 \int_0^T \int_0^1 ( \xi
    ^2e^{2s\varphi} v )_x a v_xdxdt
    \\
    &= 2 s\int_0^T \int_0^1 \xi ^2 \varphi _t e^{2s\varphi} v^2 dxdt+  2 \int_0^T \int_0^1( \xi
    ^2e^{2s\varphi})_xavv_xdxdt\\& +  2 \int_0^T \int_0^1  \xi
    ^2e^{2s\varphi}a(v_x)^{2}dxdt\\
    & = 2s \int_0^T \int_\omega \xi ^2 \varphi _t e^{2s\varphi} v^2 dxdt+  2 \int_0^T \int_\omega( \xi
    ^2e^{2s\varphi})_xavv_xdxdt\\& +  2 \int_0^T \int_\omega  \xi
    ^2e^{2s\varphi}a(v_x)^{2}dxdt.
\end{aligned}
    \]
Hence, by definition of $\xi$ and the Cauchy--Schwartz inequality,
the previous identity gives
\[
    \begin{aligned}
    &2\int_0^T \int_\omega\xi^2 e^{2s\varphi} a (v_x)^2dxdt
    =- 2s \int_0^T \int_\omega \xi^2 \varphi _t e^{2s\varphi} v^2dxdt\\
& -  2 \int_0^T \int_\omega ( \xi^2e^{2s\varphi} )_x avv_x dxdt\\
& \le - 2s\int_0^T \int_\omega \xi^2 \varphi _t e^{2s\varphi}
v^2dxdt +\int_0^T \int_\omega\left( \sqrt{a} \xi e^{s\varphi} v_x \right) ^2dxdt\\
&  + \int_0^T \int_\omega\left( \sqrt{a} \frac{( \xi^2e^{2s\varphi}
)_x}{\xi e^{s\varphi} }v \right)^2dxdt\\
& = - 2 s\int_0^T \int_\omega \xi^2 \varphi _t e^{2s\varphi}
v^2dxdt+ \int_0^T \int_\omega \xi^2  e^{2s\varphi} a(v_x)^2
dxdt\\
& + \int_0^T \int_\omega  \frac{[( \xi^2e^{2s\varphi} )_x]^2}{\xi^2
e^{2s\varphi} }av^2 dxdt.
\end{aligned}
\]
Thus,
\[\begin{aligned}
\int_0^T \int_\omega\xi^2 e^{2s\varphi} a (v_x)^2dxdt &\le - 2
\int_0^T \int_\omega \xi^2 s\varphi _t e^{2s\varphi} v^2dxdt \\
&+ \int_0^T \int_\omega  \frac{[( \xi ^2e^{2s\varphi} )_x]^2}{\xi^2
e^{2s\varphi} }av^2 dxdt.
\end{aligned}
\]
Since $x_0 \not \in \bar\omega ' $, then
\[ \begin{aligned}
&\inf_{x\in\omega'}a(x)\int_0^{T}\int _{\omega '} e^{2s\varphi}
(v_x)^2dxdt \le \int_0^T \int_{\bar{\omega}'} \xi^2
e^{2s\varphi} a (v_x)^2dxdt\\
&\le \int_0^T \int_\omega \xi^2
e^{2s\varphi} a (v_x)^2dxdt\\
&\le - 2 \int_0^T \int_\omega \xi^2
s\varphi_t e^{2s\varphi} v^2dxdt + \int_0^T \int_\omega  \frac{[(
\xi ^2e^{2s\varphi} )_x]^2}{\xi^2 e^{2s\varphi} }av^2 dxdt.
\end{aligned}
\]

Calculations show that $s\varphi_te^{2s\varphi}$ is uniformly
bounded if $s\geq s_0>0$, since $\Upsilon$ is strictly negative, a
rough estimate being
\[
|s\vp_te^{2s\varphi}|\leq c\frac{1}{s_0^{1/4}(-\max
\Upsilon)^{1/4}}.
\]
Indeed, $|\dot \Theta|\leq c \Theta^{5/4}$ and
\[
|s\vp_te^{2s\varphi}|\leq cs(-\Upsilon)\Theta^{5/4}e^{2s\vp}\leq
\frac{c}{\big(s(-\Upsilon)\big)^{5/4}}
\]
for some constants $c>0$ which may vary at every step.

On the other hand, $\displaystyle\frac{[( \xi ^2e^{2s\varphi}
)_x]^2}{\xi^2 e^{2s\varphi} }$ can be estimated by
\[
C\big(e^{2s\vp}+s^2(\vp_x)^2e^{2s\vp}\big).
\]
Of course, $e^{2s\vp}<1$, while $s^2(\vp_x)^2e^{2s\vp}$ can be
estimated with
\[
\frac{c}{(-\max \Upsilon)^2}(\Upsilon_x)^2 \leq \frac{c}{a}
\]
by \eqref{stimayx}, for some constants $c>0$.

In conclusion, we can find a positive constant $C$ such that
\[
\begin{aligned}
- 2& \int_0^T \int_\omega \xi ^2 s\varphi_t e^{2s\varphi} v^2dxdt
   + \int_0^T \int_\omega  \frac{[( \xi ^2e^{2s\varphi} )_x]^2}{\xi^2
    e^{2s\varphi} }av^2 dxdt\\
 &  \le C\int_0 ^T \int _{\omega} v^2 dxdt,
   \end{aligned}
\]
and the claim follows.
\end{proof}

We shall need the following lemma:
\begin{Lemma}\label{lemma3}
Assume Hypothesis $\ref{Ass02}$ and \eqref{omega1}. Then there exist two positive
constants $C$ and $s_0$ such that every solution $v \in \cal W$ of
\eqref{h=01} satisfies, for all $s \ge s_0$,
\[
\int_0^T\int_0^1\left( s \Theta a v_{x}^{2} + s^3 \Theta ^3
\frac{(x-x_0)^2}{a} v^{2}\right) e^{{2s\varphi}}  dxdt\le C
\int_0^T\int_{w}v^{2} dxdt.
\]
Here $\Theta$ and $\varphi$ are as in \eqref{c_1}.
\end{Lemma}

For the proof of the previous lemma we need the following classical
Carleman estimate (see, for example \cite[Proposition 4.4]{acf}):
\begin{Proposition}[\bf Classical Carleman estimates]\label{classical
Carleman}
Let $z$ be the solution of
\begin{equation}
    \label{eq-z*}
    \begin{cases}
      z_t + (a  z_x) _x= h\in L^2\big((0, T)\times (A,B)\big),
      %( a \eta _x v )_x + \eta _x a v_x =:h,&
       \\
    z(t,A)= z(t,B)=0, \; t \in (0,T),
%\begin{cases}
%z(t,0) =0, & \text{ for } (WDP), \\
%\text{ or }\\ (az_x)(t,0)=0, &\text{ for } (SDP),\\
%\end{cases}& t \in (0,T).\\
    \end{cases}
    \end{equation}
where $a\in C^1\big([A,B]\big)$ is a strictly positive function.
Then there exist positive constants $c$, $r$ and $s_0$ such that for
any $s\geq s_0$
\begin{equation}\label{carcorretta}
\begin{aligned}
& \int_0^T\int_A^B s \Theta e^{r\zeta }(z_x)^2e^{-2s\Phi}
dxdt+\int_0^T\int_A^B s^3\Theta^3 e ^{3r\zeta}z^2e^{-2s\Phi}
dxdt \\
& \le c \int_0^T\int_A^B e^{-2s\Phi} h^2 dxdt- c \int_0^T \left[
\sigma(t, \cdot) e^{-2s \Phi(t,
\cdot)}|z_x(t,\cdot)|^2\right]_{x=A}^{x=B}dt.
\end{aligned}
\end{equation}
Here the functions $\zeta, \sigma$ and $\Phi$ are defined in the
following way:
\[
\zeta (x):= \int_x^B\frac{1}{\sqrt{a(y)}}dy,\quad \sigma (t,x) := rs
\Theta(t) e^{r\zeta(x)},
\]
\[
\Phi (t,x):= \Theta(t)\Psi(x) \mbox{ and }\Psi(x):=e^{2r\zeta(A)} -
e^{r\zeta(x)}>0,
\]
where $(t,x) \in [0,T] \times[A,B]$ and $\Theta$ is defined in
\eqref{c_1}.
\end{Proposition}
(Observe that $\Phi
>0$ and $\Phi(t,x) \rightarrow + \infty$, as $ t\downarrow 0,\, t
\uparrow T$.)

\begin{proof}[Proof of Lemma $\ref{lemma3}$]
By assumption, we can find two subintervals $\omega_1\subset (0,
x_0), \omega_2 \subset (x_0,1)$ such that $(\omega_1 \cup \omega_2)
\subset \subset \omega \setminus \{x_0\}$. Now, set $\lambda_i:=
\inf \omega_i$ and $\beta_i:= \sup \omega_i$, $i=1,2$ and consider a
smooth function $\xi: [0,1] \to \Bbb R$ such that
\[\begin{cases}
    0 \leq \xi (x)  \leq 1, &  \text{ for all } x \in [0,1], \\
    \xi (x) = 1 ,  &   x \in [\lambda_1, \beta_2],\\
    %\left[\frac{2\lambda_1+\beta_1}{3}, \frac{\lambda_2+ 2 \beta_2}{3} \right], \\
    \xi (x)=0, &     x \in [0,1]\setminus \omega.
    \end{cases}\]
    Define $w:= \xi v$, where $v$ is the solution of \eqref{h=01}.
   Hence, $w$   satisfies
    \begin{equation}
    \label{eq-w*}
    \begin{cases}
      w_t + (a  w_x) _x =( a \xi _x v )_x + \xi _x a v_x =:f,&
      (t,x) \in(0, T)\times (0,1), \\
    w(t,0)= w(t,1)=0, & t \in (0,T).
    \end{cases}
\end{equation}
Applying Theorem \ref{Carleman} and using the fact that $w=0$ in a
neighborhood of $x=0$  and $x=1$, we have
\begin{equation}\label{car9}
 \int_0^T \int_0^1\Big( s \Theta  a (w_x)^2 + s^3 \Theta^3
   \frac{(x-x_0)^2}{a} w^2 \Big)
    e^{2s \varphi} \, dx dt
       \le C \int_0^T \int_0^1e^{2s \varphi} f^2  dxdt
\end{equation}
for all $s \ge s_0$. Then, using the definition of $\xi$  and in
particular the fact that  $\xi_x$ and  $\xi_{xx}$ are supported in
$\tilde \omega$, where  $\tilde \omega:= [\inf \omega, \lambda_1]
\cup[ \beta_2, \sup\omega]$, we can write
\[
f^2= (( a \xi _x v )_x + \xi _x a v_x)^2 \le C( v^2+
(v_x)^2)\chi_{\tilde \omega},
\]
since the function $a'$  is bounded on $\tilde \omega$. Hence,
applying Proposition \ref{caccio} and \eqref{car9}, we get
 \begin{equation}\label{stimacar}
\begin{aligned}
&\int_0^T\int_{\lambda_1}^{\beta_2}\left( s \Theta a (v_x)^{2} + s^3
\Theta ^3 \frac{(x-x_0)^2}{a} v^{2}\right) e^{{2s\varphi}} dxdt
\\&=\int_0^T\int_{\lambda_1}^{\beta_2}\Big( s \Theta  a (w_x)^2 + s^3
\Theta^3
   \frac{(x-x_0)^2}{a} w^2 \Big)
    e^{2s \varphi} \, dx dt\\
&\le \int_0^T \int_0^1\Big( s \Theta  a (w_x)^2 + s^3 \Theta^3
   \frac{(x-x_0)^2}{a} w^2 \Big)
    e^{2s \varphi} \, dx dt
     \\ & \le C  \int_0^T \int_{\tilde \omega}e^{2s \varphi}(
v^2+ (v_x)^2)dxdt \le C \int_0^T \int_{\omega} v^2dxdt,
\end{aligned} \end{equation} for a positive constant $C$.
Now, consider a smooth function $\eta: [0,1] \to \Bbb R$ such that
  \[\begin{cases}
    0 \leq \eta (x)  \leq 1, &  \text{ for all } x \in [0,1], \\
    \eta (x) = 1 ,  &   x \in [\beta_2 , 1],\\
    %\left[\frac{2\lambda_1+\beta_1}{3}, \frac{\lambda_2+ 2 \beta_2}{3} \right], \\
    \eta (x)=0, &     x \in \left[0,\frac{\lambda_2+ 2 \beta_2}{3}\right ].
    \end{cases}\]
    Define $z:= \eta v$, where $v$ is the solution of \eqref{h=01}.
    Then $z$ satisfies \eqref{eq-z*} and
\eqref{carcorretta}, with $h:=( a \eta _x v)_x + \eta _x a v_x$, $A=
\lambda_2$ and $B=1$. Since $h$ is supported in
$\left[\frac{\lambda_2+ 2 \beta_2}{3}, \beta_2\right]$, by
Propositions \ref{caccio} and \ref{classical Carleman} with
\begin{equation}\label{zeta1}
\zeta(x)= \zeta_1(x):=\displaystyle \int_x^1\frac{1}{\sqrt{a(y)}}dy,
\end{equation}
we get
\begin{equation}\label{lun}
\begin{aligned}
&\int_0^T\int_{\lambda_2}^1 s \Theta e^{r\zeta_1 }(z_x)^2e^{-2s\Phi}
dxdt +\int_0^T\int_{\lambda_2}^1 s^3\Theta^3 e
^{3r\zeta_1}z^2e^{-2s\Phi} dxdt\\& \le c
\int_0^T\int_{\lambda_2}^1e^{-2s\Phi} h^2 dxdt\le C \int_0^T
\int_{\tilde \omega_1} v^2dxdt + C \int_0^T \int_{\tilde
\omega_1}e^{-2s\Phi}(v_x)^2dxdt\\&\le C \int_0^T \int_{\omega}
v^2dxdt,
\end{aligned}
\end{equation}
where $\tilde \omega_1 =(\lambda_2, \beta_2)$.

Now, choose the constant $c_1$ in \eqref{c_1} so that
\[
c_1 \ge \max\left\{\frac{
e^{2r\zeta_1(\lambda_2)}-1}{c_2-\frac{(1-x_0)^2}{a(1)(2-K)}}, \frac{
e^{2r\zeta_1(\lambda_2)}-1}{c_2-\frac{x_0^2}{a(0)(2-K)}} \right\}
\]
where $\zeta_1$ is defined as before. Then, by definition of
$\varphi$, the choice of $c_1$ and by Lemma \ref{rem}, one can prove
that there exists a positive constant $k$, for example
\[k = \max \left\{\max_{\left[\lambda_2,
1\right]}a,\frac{(1-x_0)^2}{a(1)}\right\},\] such that
\[
a(x) e^{2s\varphi(t,x)} \le k e^{r\zeta_1 (x) }e^{-2s\Phi(t,x)}
\]and
\[\frac{(x-x_0)^2}{a(x)}e^{2s\varphi(t,x)} \le k e^{r\zeta_1(x)} e^{-2s
\Phi(t,x)} \le k e^{3r\zeta_1(x)} e^{-2s \Phi(t,x)}
\]
for every $(t,x) \in [0, T] \times \left[\lambda_2, 1\right]$. Thus,
by \eqref{lun}, one has
\[
\begin{aligned}
&\int_0^T\int_{\lambda_2}^1 \Big(s \Theta a (z_x)^2
 + s^3 \Theta^3
      \frac{(x-x_0)^2}{a}z^2\Big) e^{2s\varphi}dxdt \\&
\le k \int_0^T\int_{\lambda_2}^1 s \Theta e^{r\zeta_1
}(z_x)^2e^{-2s\Phi} dxdt +k\int_0^T\int_{\lambda_2}^1 s^3\Theta^3 e
^{3r\zeta_1}z^2e^{-2s\Phi} dxdt\\&
      \le kC \int_0^T
\int_{\omega} v^2dxdt,
\end{aligned}
\]
for a positive constant $C$.
%Since $v= w + z$ (recall that $w= \xi v$ and $z=\eta v$), then
%$v^{2}\le 2 (w^{2} + z^{2})$ and $v_{x}^{2}\le 2 (w_{x}^{2} +
%z_{x}^{2})$. Thus by \eqref{stimacar}, one has
%\[
%\begin{aligned}
%&\int_0^T \int_0^1\Big( s \Theta  a(x) (v_x)^2 + s^3 \Theta^3
%\frac{(x-x_0)^2}{a(x)} v^2 \Big)
%      e^{2s \varphi } \, dx dt \\
%&\le 2\int_0^T \int_0^1\Big( s \Theta  a(x) (w_x)^2 + s^3 \Theta^3
%\frac{(x-x_0)^2}{a(x)} w^2 \Big)
%      e^{2s \varphi } \, dx dt
%+ 2\int_0^T \int_0^1\Big( s \Theta  a(x) (z_x)^2 + s^3 \Theta^3
%\frac{(x-x_0)^2}{a(x)} z^2 \Big)
%      e^{2s \varphi } \, dx dt\\
%      & \le C \int_0^T \int_{\omega}v^2 dxdt +2\int_0^T\int_0^1 \Big(s \Theta a(x) (z_x)^2
% + s^3 \Theta^3
%      \frac{(x-x_0)^2}{a(x)}z^2\Big) e^{2s\varphi(t,x)}dxdt.
%\end{aligned}
%\]
As a trivial consequence,
\begin{equation}\label{stimacar2}
\begin{aligned}
&\int_0^T\int_{\beta_2}^1 \Big(s \Theta a (v_x)^2
 + s^3 \Theta^3
      \frac{(x-x_0)^2}{a}v^2\Big) e^{2s\varphi} dxdt\\&= \int_0^T\int_{\beta_2}^1 \Big(s
\Theta a (z_x)^2
 + s^3 \Theta^3
      \frac{(x-x_0)^2}{a}z^2\Big) e^{2s\varphi}dxdt\\&
      \le \int_0^T\int_{\lambda_2}^1 \Big(s \Theta a (z_x)^2
 + s^3 \Theta^3
      \frac{(x-x_0)^2}{a}z^2\Big) e^{2s\varphi}dxdt
\\&\le kC \int_0^T
\int_{\omega} v^2dxdt,
\end{aligned}
\end{equation} for a positive constant $C$.

Thus \eqref{stimacar} and \eqref{stimacar2} imply
    \begin{equation}\label{carin0}
    \begin{aligned}
      \int_{0}^T \int _{\lambda_1}^{1}  \Big( s \Theta  a (v_x)^2 + s^3 \Theta^3 \frac{(x-x_0)^2}{a}  v^2 \Big)
      e^{2s \varphi } \, dx dt
     \le  C \int_{0}^T \int _{\omega}   v^2  dxdt,
    \end{aligned}
\end{equation}
for some positive constant $C$. To complete the proof it is
sufficient to prove a similar inequality on the interval
$[0,\lambda_1]$. To this aim, we follow a reflection procedure
introducing the functions
\begin{equation}\label{W}
W(t,x):= \begin{cases} v(t,x), & x \in [0,1],\\
-v(t,-x), & x \in [-1,0],
\end{cases}
\end{equation}
where $v$ solves \eqref{h=01}, and
\begin{equation}\label{tildea}
\tilde a(x):= \begin{cases} a(x), & x \in [0,1],\\
a(-x), & x \in [-1,0].
\end{cases}
\end{equation}
 Then $W$ satisfies the problem
\begin{equation}\label{dispari}
\begin{cases}
W_t +(\tilde a W_x)_{x}= 0, &(t,x) \in  (0,T)\times (-1,1),
\\[5pt]
W(t,-1)=W(t,1) =0, & t \in (0,T).
\end{cases}
\end{equation}
Now, consider a cut off function $\rho: [-1,1] \to \Bbb R$ such that
\[
\begin{cases}
    0 \leq \rho (x)  \leq 1, &  \text{ for all } x \in [-1,1], \\
    \rho (x) = 1 ,  &   x \in (-\lambda_1, \lambda_1),\\
    %\left[\frac{2\lambda_1+\beta_1}{3}, \frac{\lambda_2+ 2 \beta_2}{3} \right], \\
    \rho (x)=0, &     x \in \left[-1,-\frac{\lambda_1+ 2\beta_1}{3}\right ]\cup \left[\frac{\lambda_1+ 2\beta_1}{3},1\right].
\end{cases}\]
Define $Z:= \rho W$, where $W$ is the solution of \eqref{dispari}.
Then $Z$ satisfies \eqref{eq-z*} and \eqref{carcorretta}, with $h:=(
\tilde a \rho_x W)_x + \rho _x \tilde a W_x$, $A= -\beta_1$ and
$B=\beta_1$. Now define
$$
\zeta(x)=\zeta_2(x):=\int_x^{\beta_1}\frac{1}{\sqrt{\tilde a(y)}}dy,
$$
Using Proposition \ref{classical Carleman} with
\begin{equation}\label{phi}
\tilde \Phi (t,x):= \Theta(t)(e^{2r\zeta_2(-\beta_1)} -
e^{r\zeta_2(x)}),
\end{equation}
the fact that $Z_x(t, -\beta_1)=Z_x(t, \beta_1)=0$, the definition
of $W$ and the fact that $\rho$ is supported in
$\left[-\frac{\lambda_1+ 2\beta_1}{3},-\lambda_1\right]
\cup\left[\lambda_1, \frac{\lambda_1+ 2\beta_1}{3}\right]$, give
\begin{equation}\label{lunga}
\begin{aligned}
&\int_0^T\int_{-\beta_1}^{\beta_1} s \Theta
e^{r\zeta_2}(Z_x)^2e^{-2s\tilde\Phi} dxdt
+\int_0^T\int_{-\beta_1}^{\beta_1} s^3\Theta^3 e
^{3r\zeta_2}Z^2e^{-2s\tilde\Phi} dxdt \\
& \le C
\int_0^T\int_{-\beta_1}^{\beta_1}e^{-2s\tilde\Phi} h^2 dxdt \\
&\le C \int_0^T \int_{-\frac{\lambda_1+
2\beta_1}{3}}^{-\lambda_1}e^{-2s\tilde\Phi}( W^2+ (W_x)^2)dxdt \\&+
C\int_0^T \int_{\lambda_1}^{\frac{\lambda_1+
2\beta_1}{3}}e^{-2s\tilde\Phi}( W^2+ (W_x)^2)dxdt.
\end{aligned}
\end{equation}
Now, putting $\Xi(x):= e^{2r\zeta_2(-\beta_1)} - e^{r\zeta_2(x)}$
and
\[
A:=\frac{\Xi(-\beta_1)}{\Xi(\beta_1)}=\frac{e^{2r\zeta_2(-\beta_1)}-e^{r
\zeta_2(-\beta_1)}}{e^{2r\zeta_2(-\beta_1)}-1}\in(0,1),
\]
we note that for any $x\in[0,\beta_1]$, $s\geq s_0$ and
$t\in(0,T)$ we have
\[
e^{-2s\Theta(t)\Xi(-x)}\leq e^{-2As\Theta (t)\Xi(x)}.
\]
Hence, using the oddness of the involved functions,
\begin{equation}\label{lunga2}
\begin{aligned}
&\int_0^T \int_{-\frac{\lambda_1+
2\beta_1}{3}}^{-\lambda_1}e^{-2s\tilde\Phi}( W^2+ (W_x)^2)dxdt \leq
\int_0^T \int^{\frac{\lambda_1+ 2\beta_1}{3}}_{\lambda_1}
e^{-2As\tilde \Phi}( W^2+ (W_x)^2)dxdt\\
& \leq \int_0^T
\int_{\lambda_1}^{\frac{\lambda_1+ 2\beta_1}{3}} v^2dxdt+  C
\int_0^T \int_{\lambda_1}^{\frac{\lambda_1+
2\beta_1}{3}}e^{-2As\theta \Xi}(v_x)^2dxdt\\
& \le  \int_0^T \int_{\omega} v^2dxdt+ C \int_0^T
\int_{\lambda_1}^{\frac{\lambda_1+ 2\beta_1}{3}}e^{-2As\theta
\Xi}(v_x)^2dxdt,
\end{aligned}
\end{equation}
for some positive constant $C$. Now, after relabeling $\tilde s=As$,
\eqref{lunga2} and Proposition \ref{caccio} imply the existence of
$C>0$ and $s_1>0$ such that for all $s\geq s_1$ we get
\begin{equation}\label{lunga3}
\int_0^T \int_{-\frac{\lambda_1+
2\beta_1}{3}}^{-\lambda_1}e^{-2s\tilde\Phi}( W^2+ (W_x)^2)dxdt \leq
C\int_0^T \int_{\omega} v^2dxdt.
\end{equation}

On the other hand, Proposition \ref{caccio} immediately implies in
an easier way that
\begin{equation}\label{lunga4}
\int_0^T \int_{\lambda_1}^{\frac{\lambda_1+
2\beta_1}{3}}e^{-2s\tilde\Phi}( W^2+ (W_x)^2)dxdt\leq C\int_0^T
\int_{\omega} v^2dxdt
\end{equation}
for all $s$ large enough and for a suitable $C>0$.

In conclusion, \eqref{lunga}--\eqref{lemma4} imply that there exists
$s_0$ and $C>0$ such that
\begin{equation}\label{dopolunga4}
\begin{aligned}
\int_0^T\int_{-\beta_1}^{\beta_1} s \Theta
e^{r\zeta_2}(Z_x)^2e^{-2s\tilde\Phi} dxdt
&+\int_0^T\int_{-\beta_1}^{\beta_1} s^3\Theta^3 e
^{3r\zeta_2}Z^2e^{-2s\tilde\Phi} dxdt\\&\leq C\int_0^T \int_{\omega}
v^2dxdt
\end{aligned}
\end{equation}
for all $s\geq s_0$.

%Since $v= w + z$ (recall that $w= \xi v$ and $z=\eta v$), then
%$v^{2}\le 2 (w^{2} + z^{2})$ and $v_{x}^{2}\le 2 (w_{x}^{2} +
%z_{x}^{2})$. Thus by \eqref{stimacar}, one has
%\[
%\begin{aligned}
%&\int_0^T \int_0^1\Big( s \Theta  a(x) (v_x)^2 + s^3 \Theta^3
%\frac{(x-x_0)^2}{a(x)} v^2 \Big)
%      e^{2s \varphi } \, dx dt \\
%&\le 2\int_0^T \int_0^1\Big( s \Theta  a(x) (w_x)^2 + s^3 \Theta^3
%\frac{(x-x_0)^2}{a(x)} w^2 \Big)
%      e^{2s \varphi } \, dx dt
%+ 2\int_0^T \int_0^1\Big( s \Theta  a(x) (z_x)^2 + s^3 \Theta^3
%\frac{(x-x_0)^2}{a(x)} z^2 \Big)
%      e^{2s \varphi } \, dx dt\\
%      & \le C \int_0^T \int_{\omega}v^2 dxdt +2\int_0^T\int_0^1 \Big(s \Theta a(x) (z_x)^2
% + s^3 \Theta^3
%      \frac{(x-x_0)^2}{a(x)}z^2\Big) e^{2s\varphi(t,x)}dxdt.
%\end{aligned}
%\]

Now, define
\[
\tilde \varphi(t,x) := \Theta(t) \tilde \psi (x),
\]
where
\begin{equation}\label{tildepsi}
\tilde \psi(x) := \begin{cases}
\psi(x), & x \ge 0,\\
\displaystyle \psi(-x)= c_1\left[\int_{-x_0}^x \frac{t+x_0}{\tilde
a(t)}dt -c_2\right], & x <0.\end{cases}
\end{equation}
and choose the constant $c_1$ so that
\[
c_1 \ge \max\left\{\frac{
e^{2r\zeta_1(\lambda_2)}-1}{c_2-\frac{(1-x_0)^2}{a(1)(2-K)}}, \frac{
e^{2r\zeta_1(\lambda_2)}-1}{c_2-\frac{x_0^2}{a(0)(2-K)}}, \frac{
e^{2r\zeta_2(-\beta_1)}-1}{c_2-\frac{(1-x_0)^2}{a(1)(2-K)}}, \frac{
e^{2r\zeta_2(-\beta_1)}-1}{c_2-\frac{x_0^2}{a(0)(2-K)}} \right\}.
\]
Thus, by definition of $\tilde \varphi$, one can prove as before
that there exists a positive constant $k$, for example
\[
k = \max \left\{\max_{\left[-\beta_1, \beta_1\right]}\tilde
a,\frac{x_0^2}{ a(0)}\right\},
\]
such that
\[
\tilde a(x) e^{2s\tilde\varphi(t,x)} \le k e^{r\zeta_2 (x)
}e^{-2s\tilde\Phi(t,x)}
\]and
\[
\frac{(x-x_0)^2}{\tilde a(x)}e^{2s\tilde\varphi(t,x)} \le k e^{r\zeta_2(x)}
e^{-2s \tilde\Phi(t,x)} \le k e^{3r\zeta_2(x)} e^{-2s\tilde
\Phi(t,x)}
\]
for every $(t,x) \,\in \,[0, T] \times \left[-\beta_1,
\beta_1\right]$. Thus, by \eqref{dopolunga4}, one has
 \begin{equation}\label{stimacar20}
\begin{aligned}
& \int_0^T\int_{-\beta_1}^{\beta_1} \Big(s \Theta \tilde a (Z_x)^2+
s^3 \Theta^3  \frac{(x-x_0)^2}{\tilde a}Z^2\Big)
e^{2s\tilde\varphi}dxdt\\
& \le k\int_0^T\int_{-\beta_1}^{\beta_1} s
\Theta e^{r\zeta_2 }(Z_x)^2e^{-2s\tilde\Phi} dxdt +
k\int_0^T\int_{-\beta_1}^{\beta_1}
s^3\Theta^3 e ^{3r\zeta_2}Z^2e^{-2s\tilde\Phi} dxdt \\
&\le kC \int_0^T \int_{\omega}v^2 dxdt.
\end{aligned}
\end{equation}
Hence, by \eqref{stimacar20} and the definition of $W$ and $Z$, we
get
\begin{equation}\label{car10}
\begin{aligned}
&\int_0^T\int_0^{\lambda_1}  \Big( s^3 \Theta^3
      \frac{(x-x_0)^2}{a}v^2+s \Theta a (v_x)^2\Big) e^{2s\varphi}dxdt\\
&= \int_0^T\int_0^{\lambda_1}  \Big( s^3 \Theta^3
      \frac{(x-x_0)^2}{a}W^2+s \Theta a (W_x)^2\Big) e^{2s\varphi} dxdt \\
&\le \int_0^T\int_{-\lambda_1}^{\lambda_1} \Big( s^3 \Theta^3
      \frac{(x-x_0)^2}{\tilde a}W^2+s \Theta \tilde a (W_x)^2\Big) e^{2s\tilde\varphi} dxdt \\&=
\int_0^T\int_{-\lambda_1}^{\lambda_1} \Big( s^3 \Theta^3
      \frac{(x-x_0)^2}{\tilde a}Z^2+s \Theta \tilde a (Z_x)^2\Big) e^{2s\tilde\varphi}dxdt \\
      &
\le \int_0^T\int_{-\beta_1}^{\beta_1} \Big( s^3 \Theta^3
      \frac{(x-x_0)^2}{\tilde a}Z^2+s \Theta \tilde a (Z_x)^2\Big) e^{2s\tilde\varphi}dxdt
      \\
&\le C \int_0^T \int_{\omega} v^2dxdt,
\end{aligned}
\end{equation}
for a positive constant $C$.

Therefore, by \eqref{carin0} and \eqref{car10}, Lemma \ref{lemma3}
follows.
\end{proof}

We shall also use the following
\begin{Lemma}\label{obser.regular}
Assume Hypothesis $\ref{Ass02}$ and \eqref{omega1}. Then there
exists a positive constant $C_T$ such that every solution $v \in
\cal W$ of \eqref{h=01} satisfies
\[
\int_0^1v^2(0,x) dx \le C_T\int_0^T \int_{\omega}v^2(t,x)dxdt.
\]
\end{Lemma}
\begin{proof}
Multiplying the equation of \eqref{h=01} by $v_t$ and integrating by
parts over $(0,1)$, one has
\[
\begin{aligned}
&0 = \int_0^1(v_t+ (av_x)_x)v_t dx= \int_0^1 (v_t^2+ (av_x)_xv_t )dx
= \int_0^1v_t^2dx + \left[av_xv_t \right]_{x=0}^{x=1} \\&-
\int_0^1av_xv_{tx} dx= \int_0^1v_t^2dx -
\frac{1}{2}\frac{d}{dt}\int_0^1a(v_x)^2
 \ge - \frac{1}{2}
\frac{d}{dt}\int_0^1 a(v_x)^2dx.
\end{aligned}
\]
Thus, the function $t \mapsto \int_0^1 a(v_x)^2 dx$ is increasing
for all $t \in [0,T]$. In particular, $\int_0^1 av_x(0,x)^2dx \le
\int_0^1av_x(t,x)^2dx$. Integrating the last inequality over
$\left[\frac{T}{4}, \frac{3T}{4} \right]$, $\Theta$ being bounded
therein, we find
%\[
%\int_0^1 a(x)(v_x)^2(0,x)dx \le \int_0^1 e^{Ct}a(x)(v_x)^2(t,x)dx \le
%e^{CT}\int_0^1a(x)(v_x)^2(t,x)dx .
%\]
%Integrating over $\left[\frac{T}{4}, \frac{3T}{4} \right]:$
\[
\begin{aligned}
\int_0^1a(v_x)^2(0,x) dx &\le
\frac{2}{T}\int_{\frac{T}{4}}^{\frac{3T}{4}}\int_0^1a(v_x)^2(t,x)dxdt\\&\le
C_T \int_{\frac{T}{4}}^{\frac{3T}{4}}\int_0^1s\Theta
a(v_x)^2(t,x)e^{2s\varphi}dxdt.
\end{aligned}
\]
Hence, by Lemma \ref{lemma3} and the previous inequality, there
exists a positive constant $C$ such that
\begin{equation}\label{stum}
\int_0^1a (v_x)^2(0,x) dx \le C \int_0^T \int_{\omega}v^2dxdt.
\end{equation}

Proceeding again as in the proof of Lemma \ref{lemma2} and applying
the Hardy- Poincar\'{e} inequality, by \eqref{stum}, one has
\[
\begin{aligned}
\int_0^1 \left(\frac{a}{(x-x_0)^2}\right)^{1/3}v^2(0,x)dx &=
\int_0^1 \frac{p}{(x-x_0)^2} v^2(0,x)dx  \\&\le C_{HP} \int_0^1
p(v_x)^2(0,x) dx \\&\le C_1C_{HP} \int_0^1a(v_x)^2(0,x) dx \le C
\int_0^T\int_{\omega}v^2dxdt,
\end{aligned}
\]
for a positive constant $C$.  Here $p(x) = (a(x)|x-x_0|^4)^{1/3}$,
$C_{HP}$ is the Hardy-Poincar\'{e} constant and $C_1:=
\max\left\{\displaystyle\left(\frac{x_0^2}{a(0)}\right)^{2/3},\displaystyle\left(\frac{(1-x_0)^2}{a(1)}\right)^{2/3}\right\}$,
as before.

By Lemma \ref{rem}, $\displaystyle \frac{a(x)}{(x-x_0)^2}$ is
nondecreasing on $[0, x_0)$ and nonincreasing on $(x_0,1]$, then
\[\left(\frac{a(x)}{(x-x_0)^2}\right)^{1/3}\ge C_2:=\min\left\{\left(\frac{a(1)}{(1-x_0)^2}\right)^{1/3},
\left(\frac{a(0)}{x_0^2}\right)^{1/3}\right\} >0.
\]
Hence
\[
C_2\int_0^1v(0,x)^2dx \le C \int_0^T\int_{\omega}v^2dxdt\] and the
thesis follows.
\end{proof}

\begin{proof}[Proof of Proposition $\ref{obser.}$] The proof is now standard, but we give it with some precise references: let $v_T \in L^2(0,1)$
and let $v$ be the solution of \eqref{h=0} associated to $v_T$.
Since $D({\cal A}^2)$ is densely defined in $L^2(0,1)$, there exists
a sequence $(v_T^n)_{n}\subset D({\cal A}^2)$ which converges to
$v_T$ in $L^2(0,1)$. Now, consider the solution $v_n$ associated to
$v_T^n$.

As shown in Theorem \ref{th-parabolic}, the semigroup
generated by $\mathcal A$ is analytic, hence $\mathcal A$ is closed (for example, see
\cite[Theorem I.1.4]{en} ; thus, by \cite[Theorem II.6.7]{en}, we
get that $(v_n)_n$ converges to a certain $v$ in $C(0,T;
L^2(0,1))$, so that
\[
\lim_{n \rightarrow + \infty} \int_0^1v_n^2(0,x) dx =
\int_0^1v^2(0,x)dx,
\]
and also
\[
\lim_{n \rightarrow + \infty}\int_0^T \int_{\omega}v_n^2dxdt =
\int_0^T \int_{\omega}v^2dxdt.
\]
But, by
Lemma \ref{obser.regular} we know that
\[
\int_0^1v_n^2(0,x) dx \le C_T\int_0^T \int_{\omega}v_n^2dxdt.
\]
Thus Proposition \ref{obser.} is now proved.
\end{proof}

\section{Linear Extension}\label{sec5}

In this section we want to extend the observability inequality
proved in the previous section starting from linear complete
problems of the form
\begin{equation}\label{linear_c}
\begin{cases}
u_t - \left(a(x)u_x \right) _x + c(t,x)u =h(t,x) \chi_{\omega}(x), & (t,x) \in (0,T) \times (0,1),\\
u(t,1)=u(t,0)=0, & t \in (0,T),\\
u(0,x)=u_0(x),& x \in (0,1),
\end{cases}
\end{equation}
where $u_0 \in L^2(0,1)$,  $h \in L^2(Q_T)$, $c \in L^\infty(Q_T)$,
$\omega$ is as in \eqref{omega1} and $a$ satisfies Hypothesis
\ref{Ass02}. Observe that the well-posedness of \eqref{linear_c}
follows by \cite[Theorem 4.1]{fggr}. As for the previous case, we
shall prove an observability inequality for the solution of the
associated homogeneous adjoint problem
\begin{equation}\label{h=0_c}
\begin{cases}
v_t +(av_x)_x- cv= 0, &(t,x) \in  (0, T)\times (0, 1),\\
v(t,1)=v(t,0)=0, & t \in (0,T),\\
v(T)= v_T \in L^2(0,1).
\end{cases}
\end{equation}
To obtain an observability inequality for \eqref{h=0_c} like the one
in Proposition \ref{obser.}, we consider the problem
\begin{equation}\label{adjoint0_c}
\begin{cases}
v_t + \left(a(x)v_x \right) _x -cv =h, & (t,x) \in (0,T) \times (0,1),\\
v(t,1)=v(t,0)=0, &  t \in (0,T),
\end{cases}
\end{equation}
and we prove the following Carleman estimate as a corollary of
Theorem \ref{Cor1}:
\begin{Corollary}\label{cor_c}
Assume Hypothesis $\ref{Ass02}$ and let $T>0$. Then, there exist two
positive constants $C$ and $s_0$, such that every solution $v$ in $
\mathcal{V}$ of \eqref{adjoint0_c} satisfies, for all $s \ge s_0$,
\[
\begin{aligned}
&\int_0^T\int_0^1 \left(s\Theta a(v_x)^2 + s^3 \Theta^3
\frac{(x-x_0)^2}{a} v^2\right)e^{2s\varphi}dxdt\\
&\le C\left(\int_0^T\int_0^1 h^{2}e^{2s\varphi}dxdt+
sc_1\int_0^T\left[a\Theta e^{2s \varphi}(x-x_0)(v_x)^2
dt\right]_{x=0}^{x=1}\right),
\end{aligned}
\]
where $c_{1}$ is the constant introduced in \eqref{c_1}.
\end{Corollary}
\begin{proof}
     Rewrite the equation of
    \eqref{adjoint0_c} as $ v_t + (av_x)_x = \bar{h}, $ where $\bar{h}
    := h + cv$. Then, applying Theorem \ref{Cor1}, there exists
    two positive constants $C$ and $s_0 >0$, such that
    \begin{equation}\label{fati1_c}
    \begin{aligned}
  &\int_0^T\int_0^1 \left(s\Theta a(v_x)^2 + s^3 \Theta^3
\frac{(x-x_0)^2}{a} v^2\right)e^{2s\varphi}dxdt\\
&\le C\left(\int_0^T\int_0^1 \bar{h}^{2}e^{2s\varphi}dxdt+
sc_1\int_0^T\left[a\Theta e^{2s \varphi}(x-x_0)(v_x)^2
dt\right]_{x=0}^{x=1}\right)
    \end{aligned}
    \end{equation}
    for all $s \ge s_0$.
    Using the definition of $\bar{h}$,
    the term $\int_0^T\int_0^1|\bar{h}|^2
    e^{2s\varphi(t,x)}dxdt$ can be estimated in the following way
    \begin{equation}\label{4_c}
    \begin{aligned}
    \int_0^T\int_0^1 \bar{h}^2e^{2s\varphi}dxdt \le
    2\int_0^T\int_0^1 h^2e^{2s\varphi}dxdt
   +2\|c\|_{L^\infty(Q_T)}^2\int_0^T\int_0^1
    e^{2s\varphi}v^2dxdt.
    \end{aligned}
    \end{equation}
Applying the Hardy-Poincar\'e inequality (see Proposition \ref{HP})
to $w(t,x) := e^{s\varphi(t,x)} v(t,x)$ and proceeding as in
\eqref{sopra}, recalling that $0<\inf \Theta\leq \Theta\leq c
\Theta^2$, one has
    \[
    \begin{aligned}
    \int_0^1 e^{2s\varphi} v^2 dx &= \int_0^1
w^2 dx\le C \int_0^1 a (w_x)^2 dx + \frac{s}{2}\int_0^1
\frac{(x-x_0)^2}{a}w^2 dx\\&\leq C\Theta\int_0^1
ae^{2s\varphi}(v_x)^2 dx+ C\Theta^3s^2 \int_0^1
    e^{2s\varphi}v^2 \frac{(x-x_0)^2}{a}dx.
    \end{aligned}
    \]
    Using this last inequality in (\ref{4_c}), we have
    \begin{equation}\label{fati2_c}
    \begin{aligned}
    \int_0^T\int_0^1 \bar{h}^2e^{2s\varphi}dxdt &\le
    %C\int_0^T\int_0^1 e^{2s\varphi(t,x)}|h|^2dxdt +C\int_0^T\int_0^1 e^{2s\varphi(t,x)}c^2 v^2dxdt\\
    %&\le
    2\int_0^T\int_0^1 |h|^2e^{2s\varphi}dxdt\\&
    + \|c\|_{L^\infty(Q_T)}^2C\int_0^T\int_0^1 \Theta a e^{2s\varphi} (v_x)^2 dxdt \\&
    + \|c\|_{L^\infty(Q_T)}^2Cs^2\int_0^T\int_0^1
     \Theta^3 e^{2s\varphi}\frac{(x-x_0)^2}{a}v^2dxdt,
    \end{aligned}
    \end{equation}
    for a positive constant $C$.
    Using this inequality
    in (\ref{fati1_c}), we obtain
    \[
    \begin{aligned}
    &\int_0^T\int_0^1 \left(s\Theta a(v_x)^2 + s^3 \Theta^3
\frac{(x-x_0)^2}{a} v^2\right)e^{2s\varphi}dxdt
    \le
    %C\int_0^T\int_0^1
    %e^{2s\varphi(t,x)}|h|^2dxdt+C\frac{\|c\|_{\infty}^2}{a(1)}\int_0^T\int_0^1
    %e^{2s\varphi(t,x)}\frac{a(x)}{x^2} v^2dxdt \\&\le
    C\Big(2\int_0^T\int_0^1|h|^2 e^{2s\varphi}dxdt\\& +
   \int_0^T\int_0^1 \Theta ae^{2s\varphi}(v_x)^2dxdt
    +s^2\int_0^T\int_0^1 e^{2s\varphi}\Theta^3
   \frac{(x-x_0)^2}{a}v^2dxdt \\
   &+sc_1\int_0^T\left[a\Theta e^{2s \varphi}(x-x_0)(v_x)^2
dt\right]_{x=0}^{x=1}\Big).
    \end{aligned}
    \]
Hence, for all $s \ge s_0$, where $s_0$ is assumed sufficiently
large, the thesis follows.
\end{proof}

As a consequence of the previous corollary, one can deduce an
observability inequality for the adjoint problem \eqref{adjoint0_c}
\eqref{h=0_c}. In fact, without loss of generality we can assume
that $c \ge 0$ (otherwise one can reduce the problem to this case
introducing $\tilde v:= e^{-\lambda t}v$ for a suitable $\lambda$).
Using this assumption we can prove that the analogous of Lemma
\ref{lemma3} and of Lemma \ref{lemme-caccio} still hold true. Thus,
as before, one can prove the following observability inequality:
\begin{Proposition}\label{obser_c}
Assume Hypotheses $\ref{Ass02}$ and \eqref{omega1}. Then there
exists a positive constant $C$ such that every solution  $v \in
C([0, T]; L^2(0,1)) \cap L^2 (0,T; H^1_a(0,1))$  of \eqref{h=0_c}
satisfies
 \begin{equation}
\int_0^1v^2(0,x) dx \le C_T\int_0^T \int_{\omega}v^2(t,x)dxdt .
\end{equation}
\end{Proposition}

\end{document}